\documentclass[11pt,reqno]{amsart}
\numberwithin{equation}{section}

\newtheorem{theorem}{Theorem}[section]
\newtheorem{proposition}[theorem]{Proposition}
\newtheorem{corollary}[theorem]{Corollary}
\newtheorem{lemma}[theorem]{Lemma}

\theoremstyle{definition}

\theoremstyle{definition} 
\newtheorem{remark}[theorem]{Remark}
\newtheorem{remarks}[theorem]{Remarks}

\newcommand{\beq}{\begin{equation}}
\newcommand{\eeq}{\end{equation}}
\newcommand{\bea}{\begin{eqnarray}}
\newcommand{\eea}{\end{eqnarray}}

\newcommand{\ZZ}{\mathbb Z}

\newcommand{\C}{\mathbb C}
\newcommand{\R}{\mathbb{R}}
\newcommand{\BB}{\mathbb{B}}
\newcommand{\EE}{\mathbb E}
\newcommand{\FF}{\mathbb F}
\newcommand{\GG}{\mathbb G}

\newcommand{\cD}{{\mathcal D}}

\newcommand{\cR}{{\mathcal R}}

\newcommand{\cDN}{{\mathcal DN}}

\newcommand{\q}{q}

\newcommand{\De}{{\mathcal D}}

\def\ab{&\hskip-3mm}
\def\eps{\varepsilon}

\newcommand{\N}{\mathbb N}
\newcommand{\Li}{\mathcal L}

\newcommand{\st}{\,;\,}

\DeclareMathSymbol{\complement}{\mathord}{AMSa}{"7B}

\def\vv<#1>{\langle #1\rangle}
\def\Vv<#1>{\bigl\langle #1\bigr\rangle}

\begin{document}
\baselineskip=15pt
\title[Two-phase Navier-Stokes equations]
{On the two-phase Navier-Stokes equations with surface tension}

\author[Jan Pr\"uss]{Jan Pr\"uss}
\address{Institut f\"ur Mathematik  \\
         Martin-Luther-Universit\"at Halle-Witten\-berg\\
         Theodor-Lieser-Str.~5\\
         D-60120 Halle, Germany}
\email{jan.pruess@mathematik.uni-halle.de}

\author[Gieri Simonett]{Gieri Simonett}
\address{Department of Mathematics\\
           Vanderbilt University
           Nashville, TN}
\email{gieri.simonett@vanderbilt.edu}

\thanks{The research of GS was partially
supported by NSF, Grant DMS-0600870.}

\begin{abstract}
The two-phase free boundary problem for the Navier-Stokes system is considered in a situation where the initial interface is close to a halfplane. By means of $L_p$-maximal regularity of the underlying linear problem we show local well-posedness of the problem, and prove that the solution, in particular the interface, becomes instantaneously real analytic.
\end{abstract}
\maketitle
\vspace{-0.4cm}
{\small\noindent
{\bf Mathematics Subject Classification (2000):}\\
Primary: 35R35. Secondary: 35Q10, 76D03, 76D45, 76T05.\vspace{0.1in}\\
{\bf Key words:} Navier-Stokes equations, surface tension,  well-posedness, analyticity. \vspace{0.2in}
\section{Introduction and Main Results}
In this paper we consider a free boundary problem
that describes the motion of two viscous incompressible capillary Newtonian fluids. 
The fluids are separated by an interface
that is unknown and has to be determined as part of the problem.
\\
Let $\Omega_{1}(0)\subset\R^{n+1}$ $(n\ge 1)$ 
be a region occupied by a viscous incompressible
fluid, $fluid_{1}$, 
and let $\Omega_{2}(0)$ be the complement of
the closure of $\Omega_{1}(0)$ in $\R^{n+1}$, corresponding
to the region occupied by a second incompressible viscous fluid,
$fluid_{2}$.
We assume that the two fluids are immiscible.
Let $\Gamma_0$ be the hypersurface that bounds $\Omega_1(0)$
(and hence also $\Omega_2(0)$) and let
$\Gamma(t)$ denote the position of $\Gamma_{0}$
at time~$t$. Thus, $\Gamma(t)$ is a sharp interface
which separates the fluids occupying the
regions $\Omega_1(t)$ and $\Omega_2(t)$, respectively,
where $\Omega_2(t):=\R^{n+1}\setminus\overline\Omega_1(t)$.
We denote the normal field on $\Gamma(t)$,
pointing from $\Omega_1(t)$ into $\Omega_2(t)$, by $\nu(t,\cdot)$.
Moreover, we denote by
$V(t,\cdot)$ and $\kappa(t,\cdot)$ the normal velocity and
the mean curvature of $\Gamma(t)$ with respect to $\nu(t,\cdot)$,
respectively. Here the curvature $\kappa(x,t)$ is assumed to be negative
when $\Omega_1(t)$ is convex in a neighborhood of $x\in\Gamma(t)$.
The motion of the  fluids is governed by the following system of
equations for $i=1,2:$
\begin{equation}
\label{NS-2phase}
\left\{
\begin{aligned}
\rho_i\big(\partial_tu+(u|\nabla)u\big)
                   -\mu_i\Delta u+\nabla \q
                       & = 0 &\ \hbox{in}\quad &\Omega_i(t)\\
    {\rm div}\,u & = 0 &\ \hbox{in}\quad &\Omega_i(t) \\
           -{[\![S(u,\q)\nu]\!]}& = \sigma\kappa \nu &\ \hbox{on}\quad &\Gamma(t)\\
            {[\![u]\!]}& = 0 &\ \hbox{on}\quad &\Gamma(t) \\
                      V& = (u| \nu) &\ \hbox{on}\quad &\Gamma(t)\\
             u(0)& = u_0&\ \hbox{in}\quad &\Omega_i(0)\\
              \Gamma(0)& = \Gamma_0\,.\\
\end{aligned}
\right.
\end{equation}
Here, $S=S(u,q)$ is the stress tensor defined by
\begin{equation*}
S(u,\q)=\mu_i\big(\nabla u+(\nabla u)^{\sf T}\big)-qI
\quad \text{in}\quad \Omega_i(t),
\end{equation*}
and
\begin{equation*}
[\![v]\!]=(v_{|_{\Omega_2(t)}}-v_{|_{\Omega_1(t)}}\big)|_{\Gamma(t)}
\end{equation*}
denotes the jump of the quantity $v$, defined on the respective
domains $\Omega_i(t)$, across the interface $\Gamma(t)$.
\smallskip\\
Given are the initial position $\Gamma_{0}$
of the interface, and the initial velocity
$$
u_0:\Omega_0\to \R^{n+1},\quad \Omega_0:=\Omega_1(0)\cup\Omega_2(0).
$$
The unknowns are the velocity field
$u(t,\cdot):\Omega(t)\to \R^{n+1}$,
the pressure field $\q(t,\cdot):\Omega(t)\to \R$,
and the free boundary $\Gamma(t)$,
where $\Omega(t):=\Omega_1(t)\cup\Omega_2(t)$.
\\
The constants $\rho_i>0$ and $\mu_i>0$ denote the densities
and the viscosities of the respective fluids,
and the constant $\sigma$ stands for the surface tension.
Hence the material parameters $\rho_i$ and $\mu_i$ depend
on the phase $i$, but otherwise are assumed to be constant.
System \eqref{NS-2phase} comprises the
{\em two-phase Navier-Stokes equations with surface tension}.
The first equation in \eqref{NS-2phase} reflects balance of momentum,
while the second expresses the fact that both fluids are incompressible.
If surface tension is neglected, the  boundary condition
on $\Gamma(t)$ would be the equality of stress on the two sides
of the surface. The effect of surface tension
introduces a discontinuity in the normal component
of $[\![S(u,q)]\!]$ proportional to the
mean curvature of $\Gamma(t)$.
The forth equation stipulates that the velocities are continuous
across $\Gamma(t)$.
Finally, the fifth equation, called the kinematic boundary condition,
expresses that fluid particles cannot cross $\Gamma(t)$.
\\
In order to economize our notation, we set
\begin{equation*}
\rho=\rho_1\chi_{\Omega_1(t)}+ \rho_2\chi_{\Omega_2(t)},
\quad
\mu=\mu_1\chi_{\Omega_1(t)}+\mu_2\chi_{\Omega_2(t)},
\end{equation*}
where $\chi_D$ denotes the indicator function of a set $D$.
With this convention, system \eqref{NS-2phase} can be recast as
\begin{equation}
\label{NS-2}
\left\{
\begin{aligned}
 \rho\big(\partial_tu+(u|\nabla)u\big)-\mu\Delta u+\nabla \q
                   &=  0\ &\hbox{in}\quad &\Omega(t)\\
      {\rm div}\,u & = 0\ &\hbox{in}\quad &\Omega(t) \\
     -{[\![S(u,\q)\nu]\!]}& = \sigma\kappa \nu &\hbox{on}\quad &\Gamma(t)\\
       {[\![u]\!]} & = 0\ &\ \hbox{on}\quad &\Gamma(t) \\
                  V& = (u|\nu) &\ \hbox{on}\quad &\Gamma(t)\\
               u(0)& = u_0&\ \hbox{in}\quad &\Omega_{0}\\
          \Gamma(0)& = \Gamma_0\,.\\
\end{aligned}
\right.
\end{equation}
In this publication we consider the case where $\Gamma_0$ 
is a graph over $\R^n$ given by a function $h_0$.
We then set $\Omega_1(0)=\{(x,y)\in\R^n\times\R: y<h_0(x)\}$
and consequently, $\Omega_2(0)=\{(x,y)\in\R^n\times\R: y>h_0(x)\}$.  
Our main result on existence, uniqueness, and regularity
of solutions then reads as follows.
\medskip
\begin{theorem}
\label{th:1.2}
Suppose $p>n+3$.
Then given $t_0>0$, there exists $\eps_0=\eps_0(t_0)>0$ such
that for any initial values
$$(u_0, h_0)\in W^{2-2/p}_p(\Omega_0,\R^{n+1})
         \times W^{3-2/p}_p(\R^n),
$$
satisfying the compatibility conditions
$$
[\![\mu D(u_0)\nu_0-\mu (\nu_0| D(u_0)\nu_0)\nu_0]\!]=0,
\quad {\rm div} \; u_0=0\; \text{ on }\; \Omega_0,
\quad [\![u_0]\!]=0,
$$
with $D(u_0):=(\nabla u_0+(\nabla u_0)^{\sf T})$,
and the smallness condition
$$
\|u_0\|_{W^{2-2/p}_p(\Omega_0)}
+ \|h_0\|_{W^{3-2/p}_p(\R^n)}\le \eps_0
$$
problem \eqref{NS-2}
admits a classical solution $(u,q,\Gamma)$ on $(0,t_0)$.
The solution is unique in the function class described in
Theorem~\ref{th:nonlinearII}.
In addition, $\Gamma(t)$
is a graph over $\R^n$ given
by a function $h(t)$,
$ {\mathcal M}=\bigcup_{t\in(0,t_0)}\big(\{t\}\times\Gamma(t)\big)$
is a real analytic manifold, and with
$$
\mathcal{O}\:=\{(t,x,y):\; t\in(0,t_0),\; x\in\R^n,
y\neq h(t,x)\},
$$
the function $(u,\q):{\mathcal O}\rightarrow\R^{n+2}$ is  real analytic.
\end{theorem}
\begin{remarks}
(a) Theorem 1.1 shows that solutions immediately regularize and 
become analytic in space and time.
If one thinks of the situation of oil in contact with water,
this result seems plausible, as capillary forces
tend to smooth out corners in the interface separating
the two different fluids.
\smallskip\\
(b)
More precise statements for a transformed version of problem  
\eqref{NS-2} will be given in Section 6. 
Due to the restriction $p>n+3$, we shall show that
\begin{equation}
\label{h-continuous}
h\in C(J;BU\!C^2(\R^n))\cap C^1(J;BU\!C^1(\R^n))
\end{equation}
where $J=[0,t_0]$. In particular, the normal of $\Omega_1(t)$, the
normal velocity of $\Gamma(t)$, and the mean curvature of $\Gamma(t)$
are well-defined and continuous, so that \eqref{NS-2} makes sense
pointwise. 
For $u$ and $q$ we obtain
\begin{equation}
\label{u-continuous}
\begin{split}
&u(t,\cdot)\in BU\!C^1(\Omega(t),\R^{n+1})\;\;\text{for}\; t\in J,
\quad u\in BU\!C(J\times\R^{n+1},\R^{n+1}), 
\\
&q(t,\cdot)\in U\!C(\Omega(t))\;\; \text{for}\; t\in J\setminus\{0\}.
\end{split}
\end{equation}
In addition, the solution $(u,\q,h)$ depends continuously on the 
initial values $(u_0,h_0)$.
Also interesting is the fact that the surface pressure jump
will turn out to be real analytic as well.
\smallskip\\
(c)
It is possible to relax the assumption $p>n+3$. In fact, $p>(n+3)/2$ can be shown to be sufficient. However, to keep the arguments as simple as possible, here we impose the stronger condition $p>n+3$.
\smallskip\\
(d)
If gravity acts on the fluids then the condition on the free boundary 
is to be replaced by
\begin{equation}
\label{gravitity}
-[\![S(u,q)]\!]\nu=\sigma H\nu+\gamma [\![\rho]\!]y \nu\quad\text{on}\quad\Gamma(t),
\end{equation}
where $y$ denotes the vertical component of a generic point on 
$\Gamma(t)$, and where $\gamma>0$ is the gravity acceleration.
It is not very difficult to verify that our approach also covers this case,
yielding a solution having the same regularity properties
as stated in the theorem above, provided
$\rho_2\le\rho_1$, i.e. the heavier fluid lies beneth the lighter one.
Indeed, an analysis of our proof shows that we only need to replace 
the symbol $s(\lambda,\tau)$ introduced in~\eqref{s-lambda-tau} by
\begin{equation*}
 s(\lambda,\tau)=\lambda + \sigma\tau k(z)-\frac{\gamma [\![\rho]\!]}{\tau}k(z). 
\end{equation*}
It satisfies the same estimates as in 
\eqref{estsymb} in case that $\rho_2\le\rho_1$.
\smallskip\\
(e) We mention that our results also cover
the one phase Navier-Stokes equations with surface tension
\eqref{NS-1phase}.  
\smallskip\\
(f) 
The solutions we obtain exist on an interval $(0,t_0)$ 
with $t_0>0$ arbitrary, but fixed,
provided the initial data are sufficiently small.
It can be shown that problem~\eqref{NS-2} also
admits unique local solutions 
that enjoy the same regularity properties as above, provided 
$\sup_{x\in\R}|\nabla h_0|$ is sufficiently small
in relation to the horizontal component of $u_0$.
In this case, no other smallness conditions on the data are required.
The proof of this result is considerably more involved,
and the analysis requires delicate estimates
for the nonlinear terms.
Additionally, we need a modified version of Theorem 5.1
in order to dominate some of the nonlinear terms by linear ones.
The proof of this modification will involve 
introducing a countable partition of unity and
then establishing commutator estimates for certain pseudo-differential operators.  
Since this paper is already rather long, we 
refrain from including a proof of this result here.
It will be contained in the forthcoming paper~\cite{PrSi09b}.
\end{remarks}
\smallskip
\noindent
Let us now discuss and contrast our results with results
previously obtained by other researchers.
In case $\Omega_2(t)=\emptyset$ one obtains the
{\em one-phase Navier-Stokes equations with surface tension}
\begin{equation}
\label{NS-1phase}
\left\{
\begin{aligned}
  \rho\big(\partial_tu +(u|\nabla)u\big)-\mu \Delta u +\nabla q
                     & = 0 &\ \hbox{in}\quad &\Omega (t)\\
        {\rm div}\,u & = 0 &\ \hbox{in}\quad &\Omega (t) \\
            {S}(u,\q)\nu & = \sigma\kappa\nu &\ \hbox{on}\quad &\Gamma(t) \\
                    V& = (u|\nu) &\ \hbox{on}\quad &\Gamma(t)\\
                 u(0)& = u_0&\ \hbox{in}\quad &\Omega_{0}\\
            \Gamma(0)& = \Gamma_0\,.\\
\end{aligned}
\right.
\end{equation}
Equations \eqref{NS-1phase}
describe the motion of an isolated
liquid which moves due to capillary forces acting on the free boundary.
\smallskip\\
Problem \eqref{NS-1phase} has received wide attention 
in the last two decades or so.
Existence and uniqueness of solutions
for $\sigma>0$, as well as for $\sigma=0$,
in case that $\Omega(0)$ is bounded 
(corresponding to an isolated fluid drop) has been
extensively studied in a long series of papers by Solonnikov,
see for instance \cite{So84}--\cite{So03b} and \cite{MoSo92} 
for the case $\sigma>0$.
Solonnikov proves existence and uniqueness results in various function
spaces, including anisotropic H\"older and Sobolev-Slobodetskii spaces.
Moreover, it is shown in \cite{So86} that if $\Omega_{0}$ is sufficiently close
to a ball and the initial velocity $u_{0}$ is sufficiently small,
then the solution exists globally, and converges to a uniform rigid rotation 
of the liquid about a certain axis which is moving uniformly with a constant speed,
see also \cite{PaSo02}.
More recently, local existence and uniqueness
of solutions for \eqref{NS-1phase} 
(in case that $\Omega$ is a bounded domain,
a perturbed infinite layer, or a perturbed half-space)
in  anisotropic Sobolev spaces $W^{2,1}_{p,q}$
with $2<p<\infty$ and $n<q<\infty$ has been
established by Shibata and Shimizu in \cite{SS09, SS09b}.
For results concerning \eqref{NS-1phase} with $\sigma=0$
we refer to the recent contributions 
\cite{SS07, SS08} and the references therein.
\smallskip\\
The motion of a layer of  viscous, 
incompressible fluid in an ocean of infinite extent,
bounded below by a solid surface and above by a free surface which includes
the effects of surface tension and gravity
(in which case $\Omega_0$ is a strip, bounded above
by $\Gamma_0$ and below by a fixed surface $\Gamma_b$)
is considered by
Allain~\cite{Al87}, 
Beale~\cite{Bea84}, Beale and Nishida~\cite{BeaNi84},
Tani~\cite{Ta96}, and by Tani and Tanaka~\cite{TT95}.
If the initial state and the initial velocity are close
to equilibrium, global existence of solutions is proved
in \cite{Bea84} for $\sigma>0$,
and in \cite{TT95} for $\sigma\ge 0$,
and the asymptotic decay rate for $t\to\infty$
is studied in~\cite{BeaNi84}.
\smallskip\\
Results concerning the {\em two-phase problem}~\eqref{NS-2} are more recent.
Existence and uniqueness of local solutions 
is studied in \cite{Deni91, Deni94, DS95,Tanaka93}.
In more detail, Densiova~\cite{Deni94} 
establishes existence and uniqueness of solutions
(of the transformed problem in Lagrangian coordinates)
with $v\in W^{r,r/2}_2$ 
for $r\in (5/2,3)$
in case that one of the domains is bounded.
Tanaka \cite{Tanaka93} considers the two-phase Navier-Stokes equations
with thermo-capillary convection in bounded domains, and he obtains
existence and uniqueness of solutions 
with $(v,\theta)\in W^{r,r/2}_2$ 
for $r\in (7/2,4)$,
with $\theta$ denoting the temperature.
\smallskip\\
The approach used by Solonnikov, and also
in \cite{Deni91}--\cite{DS95}, \cite{SS07,SS08,SS09,SS09b,Tanaka93,Ta96,TT95},
relies on a formulation of systems \eqref{NS-2}
and \eqref{NS-1phase} in Lagrangian coordinates.
In this formulation one obtains a transformed
problem for the velocity and the pressure on a fixed domain, 
where the free boundary does not occur explicitly. The free boundary 
can then be obtained  by
$$
\Gamma(t)=\big\{\xi+\int_{0}^{t}v(\tau,\xi)d\tau: \xi\in \Gamma_0\big\},
$$
where $v$ is the velocity field in Lagrangian coordinates.
It is not clear whether this formulation allows one to
obtain smoothing results for the free boundary,
as the regularity of $\Gamma(t)$ seems to be restricted
by the regularity of $\Gamma_0$.
To the best of our knowledge, the regularity of the
free boundary for the Navier-Stokes equations 
with surface tension \eqref{NS-2phase} or \eqref{NS-1phase}
has not been addressed in the literature before,
with the notable exception of \cite{Bea84}.
Beale considers the ocean problem
with $\Omega(t)=\{(x,y)\in \R^2\times\R: -b(x)< y< h(t,x)\}$
and he shows by a boot-strapping argument that solutions
are $C^k$ for any given fixed $k\in\N$,
where the size of the initial data must be adjusted
in dependence of $k$.
As in our case, his approach does not rely
on a formulation in Lagrangian coordinates.

In order to prove our main result we 
transform problem \eqref{NS-2} into a problem on a fixed domain.
The transformation is expressed in terms of the
unknown height function $h$ describing the free boundary.
Our analysis proceeds with studying solvability properties
of some associated linear problems.
It is important to point out that we succeed in establishing
optimal solvability results 
(also referred to as as maximal regularity),
see Theorem~\ref{th:3.1}, Proposition~\ref{pro:3.3},
Theorem~\ref{th:4.1}, Corollary~\ref{co:4.2} and Theorem~\ref{th:5.1}.
In other words, we show that the linear problems define
an isomorphism between properly chosen function spaces.
This property, in turn, allows us to resort to the
implicit function theorem to establish the analyticity of 
solutions to the nonlinear problem, as will be pointed out below.
All our results for the associated linear problems mentioned above 
seem to be new, 
as they give sufficient as well as necessary conditions for solvability.
Our analysis is greatly facilitated by 
studying the Dirichlet-to-Neumann operator for the
Stokes equations, see Section~4.
It is interesting, and maybe even surprising,
to observe the mapping properties of this operator, see Theorem 4.1.
Our approach for establishing solvability results
relies on the powerful theory of maximal regularity,
in particular on the  $H^\infty$-calculus for sectorial
operators, the Dore-Venni theorem, and the Kalton-Weis theorem, 
see for instance~\cite{DHP03, KW01, KuWe04, PrSi06}.

Based on the linear estimates we can solve the nonlinear problem
by the contraction mapping principle.
Analyticity of the solution is obtained in a rather short and
elegant way by the implicit function theorem in conjunction
with a scaling argument, relying on an idea that goes back
to Angenent \cite{Ang90a, Ang90b} and Masuda \cite{Ma80};
see also \cite{ES96, ES03, EPS03b}. 
More precisely, by introducing parameters which represent scaling in time,
and translation in space, the implicit function theorem
yields analytic dependence of the solution 
of a parameter dependent-problem on the parameters,
and this can be translated to a smoothness result
in space and time for the original problem.

The plan for this paper is as follows. Section 2 contains the
transformation of the problem to a half-space and the determination
of the proper underlying linear problem.
In Sections 3, 4 and 5 we study this linearization and prove in particular the
crucial maximal regularity results in an $L_p$-setting. Section~6
is then devoted to the nonlinear problem and contains the proof
of our main result.
\section{Reduction to a Flat Interface}
In this section we first transform the free boundary problem \eqref{NS-2}
to a fixed domain, and we then introduce some
function spaces that will be used throughout the paper.
Suppose that $\Gamma(t)$ is a graph over $\R^n$, parametrized as
$$
\Gamma(t)=\{(x,h(t,x)): \; x\in\R^n\},\quad t\in J,
$$
with $\Omega_2(t)$ lying ``above'' $\Gamma(t)$, i.e.
$
\Omega_2(t)=\{ (x,y)\in\R^n\times\R: y>h(t,x)\}$ for $t\in J:=[0,a]$.
Reduction from deformed into true halfspaces is achieved by means
of the transformations
\bea
&&v(t,x,y)=\left[\begin{array}{l}
                  u_1(t,x,h(t,x)+ y)\\
                  \vdots\\
                  u_n(t,x,h(t,x)+ y)
                 \end{array}  \nonumber\right],\\
&&w(t,x,y)=u_{n+1}(t,x,h(t,x) + y),\nonumber\\
&&\pi(t,x,y)=\q(t,x,h(t,x)+ y),\nonumber
\eea
where $t\in J$, $x\in \R^n$, $y\in\R$, $y\neq0$.
Since for $j,k=1,\ldots,n$ we have
\begin{equation}
\begin{aligned}
\partial_ju_k &=\partial_jv_{ k}-\partial_jh\partial_yv_{ k},\quad
\partial_{n+1}u_k =\partial_yv_{ k}, \\
\partial_ju_{n+1} &=\partial_jw -\partial_jh\partial_yw,\quad
\partial_{n+1}u_{n+1} =\partial_yw, \\
\partial_j\q &=\partial_j\pi -\partial_jh\partial_y\pi,\quad
\partial_{n+1}\q =\partial_y\pi,  \\
\partial_t u_k &=\partial_tv_{k} -\partial_th\partial_y v_{ k},\quad
\partial_t u_{n+1} =\partial_t w -\partial_th\partial_y w,
\end{aligned}
\end{equation}
and
\begin{equation*}
\begin{aligned}
&\Delta u_k=\Delta_{x}v_{ k}- 2(\nabla h|\nabla_{x})\partial_yv_{ k} +(1+|\nabla h|^2)
\partial^2_yv_{ k}-\Delta h\partial_yv_{k}, \\
&\Delta u_{n+1}=\Delta_{x}w- 2(\nabla h|\nabla_{x})\partial_yw +(1+|\nabla h|^2)
\partial^2_yw-\Delta h\partial_yw,
\end{aligned}
\end{equation*}
we obtain from \eqref{NS-2} the following quasilinear system with initial
conditions
\begin{equation}
\label{2.1}
\left\{
\begin{aligned}
\rho\partial_t v-\mu\Delta_{x} v-\mu\partial^2_y v+\nabla_{x}\pi  &=F_v(v,w,\pi,h)  &\ \hbox{in}\quad &(0,\infty)\times\dot\R^{n+1}\\
\rho\partial_t w-\mu\Delta_{x} w-\mu\partial^2_y w+\partial_y\pi
&=F_w(v,w,h)&\ \hbox{in}\quad &(0,\infty)\times\dot\R^{n+1}\\
{\rm div}_{x}v+\partial_y w &=F_d(v,h) &\ \hbox{in}\quad &(0,\infty)\times \dot\R^{n+1}\\
v(0,x,y)=v_0(x,y),\; w(0,x,y)&=w_0(x,y) &\ \hbox{in}\quad &\dot\R^{n+1}
\end{aligned}
\right.
\end{equation}
where $\dot \R^{n+1}=\{(x,y)\in\R^n\times\R\st y\neq 0\}$.
Here and in the sequel, $\nabla h$ and $\Delta h$ always
denote the gradient and the Laplacian of $h$ with respect to $x\in\R^n$.
Note that $\rho$ and $\mu$ in general have jumps at $y=0$, i.e.\
$\rho=\rho_2$ for $y>0$, $\rho=\rho_1$ for $y<0$, and similarly for $\mu$.
The nonlinearities are given by
\begin{equation}
\label{2.2}
\begin{split}
F_v(v,w,\pi,h)&=\mu\{- 2(\nabla h|
\nabla_{x})\partial_yv +|\nabla h|^2
\partial^2_yv-\Delta h\partial_yv\}+\partial_y \pi\nabla h \\
&\quad
+\rho\{-(v|\nabla_{x})v
+(\nabla h|v)\partial_yv-
w\partial_yv\}
+ \rho\partial_th \partial_y v, \\
F_w(v,w,h)&=\mu\{- 2(\nabla h|
\nabla_{x})\partial_yw +|\nabla h|^2
\partial^2_yw -\Delta h\partial_yw\} \\
&\quad  +\rho\{-(v|\nabla_{x})w
+(\nabla h|v)\partial_yw- w\partial_yw\}
+ \rho\partial_th \partial_y w, \\
F_d(v,h)&= (\nabla h|\partial_y v).
\end{split}
\end{equation}
Note that these functions are  polynomials in the derivatives of
$(v,w,\pi,h)$, hence analytic, and linear with respect to second
derivatives, with coefficients of first order. This exhibits the quasilinear
character of the problem.

To obtain the transformed interface conditions we observe that
the outer normal $\nu$ of $\Omega_1(t)$ is given by
\begin{equation*}
\nu(t,x)=\frac{1}{\sqrt{1+|\nabla h(t,x)|^2}}
               \left[\begin{array}{c}
                     -\nabla h(t,x)\\ 1
                     \end{array}\right],
\end{equation*}
where, as above, $\nabla h(t,x)$ denotes the
gradient vector of $h$ with respect to $x\in\R^n$.
The normal velocity $V$ of $\Gamma(\cdot)$ is
\begin{equation*}
V(t,x)=\partial_th(t,x)/\sqrt{1+|\nabla h(t,x)|^2}.
\end{equation*}
The kinematic condition $V=(u|\nu)$ on $\Gamma(\cdot)$ now reads as
\begin{equation}
\label{H}
\partial_t h-\gamma w=H(v,h),\qquad H(v,h):=-(\gamma v| \nabla h).
\end{equation}
Here $(\gamma w)(x):=w(x,0)$ denotes the trace of
the function $w:\dot \R^{n+1}\to \R$ and, correspondingly,
$\gamma v$ is the trace of
$v:\dot\R^{n+1}\to \R^n$.
Since $u$ is continuous across $\Gamma(t)$,
$\gamma v$ and $\gamma w$ are unambiguously defined.
It is also noteworthy to observe that
the tangential derivatives of
$v$ and $w$ are continuous across $\R^n$.
The curvature of $\Gamma(t)$ is given by
\begin{equation*}
\label{2.4}
\kappa(t,x)={\rm div}_x\left(
\frac{\nabla h(t,x)}
{\sqrt{1+|\nabla h(t,x)|^2}}\right)
=\Delta h-G_\kappa(h),
\end{equation*}
see for instance equation (24) in \cite[Appendix]{BPS05},
with
\begin{equation}
\label{2.41}
G_\kappa(h)=\frac{|\nabla h|^2\Delta h}
                {(1+\sqrt{1+|\nabla h|^2})
                \sqrt{1+|\nabla h|^2}}
   +\frac{(\nabla h| \nabla^2 h\nabla h)}
                {(1+|\nabla h|^2)^{3/2}},
\end{equation}
where $\nabla^2 h$ denotes the Hessian matrix of all second
order derivatives of $h$.
The components of $\De(v,w,h)$, the transformed version of the deformation tensor
$D(u)=(\nabla u+(\nabla u)^{\sf T})$, are given by
\begin{equation}
\label{D}
\begin{split}
 \De_{ij}(v,w,h)=\ & \partial_iv_{ j}+\partial_jv_{ i}
-(\partial_ih\partial_yv_{ j}+\partial_j h\partial_yv_{ i}),
\\
 \De_{n+1,j}(v,w,h)=\ &\De_{j,n+1}(v,w,h)
 =\partial_y v_{ j}+\partial_jw-\partial_jh \partial_yw,\\
 \De_{n+1,n+1}(v,w,h)=\ & 2\partial_y w,
 \end{split}
\end{equation}
for $i,j=1,\ldots,n$, where $\delta_{ij}$ denotes the Kronecker symbol.
For the jumps of the components of the deformation tensor this yields
\begin{equation*}
\begin{split}
[\![\mu\De_{ij}(v,w,h)]\!] =\ &
[\![\mu(\partial_iv_{j}+\partial_jv_{ i})]\!]
-\partial_i h[\![\mu\partial_yv_{j}]\!]
-\partial_j h[\![\mu\partial_yv_{ i}]\!],
\\
[\![\mu\De_{n+1,j}(v,w,h)]\!] =\ &[\![\mu\De_{j,n+1}(v,w,h)]\!]
=[\![\mu\partial_j w]\!]+[\![\mu\partial_yv_{j}]\!]-\partial_jh
[\![\mu\partial_yw]\!],\\
 [\![\mu\De_{n+1,n+1}(v,w,h)]\!] =\ &2[\![\mu\partial_y w]\!].
\end{split}
\end{equation*}
Therefore, the jump condition for the normal stress at the interface yields the
following boundary conditions:
\begin{equation}
\label{2.6}
\begin{split}
 -[\![\mu\partial_yv]\!]-[\![\mu\nabla_x w]\!]
  =\  &G_v(v,w,[\![\pi]\!],h),\\
 -2[\![\mu\partial_yw]\!]+[\![\pi]\!]-\sigma\Delta h
 =\ &G_w(v,w,h),
\end{split}
\end{equation}
where the nonlinearities $(G_v,G_w)$ have the form
\begin{equation}
\label{2.7}
\begin{split}
G_v(v,w,[\![\pi]\!],h)&\!=\!
-[\![\mu(\nabla_{x}v+(\nabla_{x}v)^{\sf T})]\!]\nabla h
+|\nabla h|^2[\![\mu\partial_yv]\!]
+ (\nabla h|\,[\![\mu\partial_yv]\!])\nabla h\\
&\quad -[\![\mu\partial_y w]\!]\nabla h
 + \{[\![\pi]\!]-\sigma(\Delta h-G_\kappa(h))\}\nabla h, \\
\ \ G_w(v,w,h)&\!=\!
- (\nabla h|\, [\![\mu\partial_yv]\!])
- (\nabla h|\, [\![\mu\nabla_{x}w]\!])
+ |\nabla h|^2  [\![\mu\partial_yw]\!]
-\sigma G_\kappa(h).
\end{split}
\end{equation}
We note that $G=(G_v,G_w)$ is anaytic in $(v,w,[\![\pi]\!],h)$.
Moreover, $G$ is linear in $(v,w,[\![\pi]\!])$, and in the second derivatives of $h$.
Thus the boundary conditions are quasilinear as well.

Summarizing, we arrive at the following
problem  for $u=(v,w)$, $\pi$, and $h$:
\begin{equation}
\label{tfbns2}
\left\{
\begin{aligned}
\rho\partial_tu -\mu\Delta u+\nabla \pi&= F(u,\pi,h)
    &\ \hbox{in}\quad &\dot\R^{n+1}\\
{\rm div}\,u&= F_d(u,h)&\ \hbox{in}\quad &\dot\R^{n+1}\\
-[\![\mu\partial_y v]\!] -[\![\mu\nabla_{x}w]\!] &=G_v(u,[\![\pi]\!],h)
    &\ \hbox{on}\quad &\R^n\\
-2[\![\mu\partial_y w]\!] +[\![\pi]\!] -\sigma\Delta h &= G_w(u,h)
    &\ \hbox{on}\quad &\R^n\\
[\![u]\!] &=0 &\ \hbox{on}\quad &\R^n\\
\partial_th-\gamma w&=H(u,h) &\ \hbox{on}\quad &\R^n\\
u(0)=u_0,\; h(0)&=h_0, \\\end{aligned}
\right.
\end{equation}
for $t>0$.
This is problem \eqref{NS-2} transformed to the half-spaces
$\R^{n+1}_\pm:=\{(x,y)\in\R^n\times\R: \pm y>0\}$. 
\medskip\\
Before studying solvability results for problem \eqref{tfbns2}
let us first introduce suitable function spaces.
Let $\Omega\subseteq\R^m$ be open and $X$ be
an arbitrary Banach space.
By $L_p(\Omega;X)$ and $H^s_p(\Omega;X)$,
for $1\le p\le\infty$, $s\in\R$, we denote the $X$-valued Lebegue and
the Bessel potential spaces of order $s$, respectively.
We will also frequently make use of the fractional
Sobolev-Slobodeckij
spaces $W^s_p(\Omega;X)$, $1\le p< \infty$,
$s\in\R\setminus\ZZ$, with norm
\begin{equation}
\label{Slobodeskii}
         \|g\|_{W^s_p(\Omega;X)}
         =\|g\|_{W^{[s]}_p(\Omega;X)}
         +\sum_{|\alpha|=[s]}
          \left(\int_\Omega\int_\Omega
          \frac{\|\partial^\alpha g(x)-\partial^\alpha g(y)\|^p_{X}}
          {|x-y|^{m+(s-[s])p}}\, dx\, dy\right)^{\!1/p}{\hskip-4mm},
\end{equation}
where $[s]$ denotes the largest integer smaller than $s$.
Let $a\in(0,\infty]$ and $J=[0,a]$.
We set
\[
         _0W^s_p(J;X):=\left\{
         \begin{array}{l}
         \{g\in W^s_p(J;X): g(0)=g'(0)=\ldots=g^{(k)}(0)=0\},\\[3mm]
         \mbox{if}\quad
         k+\frac1{p}<s<k+1+\frac1{p}, \ k\in\N\cup\{0\},\\[3mm]
         W^s_p(J;X),\quad \mbox{if}\quad s<\frac1{p}.
         \end{array}
         \right.
\]
The spaces $_0H^s_p(J;X)$ are defined analogously.
Here we remind that $H^k_p=W^k_p$ for $k\in\ZZ$ and $1<p<\infty$,
and that $W^s_p=B^s_{pp}$ for $s\in\R\setminus\ZZ$.
\\
For $\Omega\subset\R^m $ open and $1\le p<\infty$,
the { homogeneous Sobolev spaces} $\dot H^1_p(\Omega)$ of order $1$
are defined as
\begin{equation}
\label{dot-H-1}
\begin{split}
&\dot H^1_p(\Omega):=(\{g\in L_{1,\text{loc}}(\Omega):
\|\nabla g\|_{L_p(\Omega)}<\infty\},\|\cdot\|_{\dot H^1_p(\Omega)}) \\
&\|g\|_{\dot H^1_p(\Omega)}:=
\big(\sum\limits_{j=1}^m \|\partial_j g\|^p_{L_p(\Omega)}\big)^{\!1/p}.
\end{split}
\end{equation}
Then $\dot H^1_p(\Omega)$
is a Banach space, provided we factor out the constant functions
and equip the resulting space with the corresponding quotient norm,
see for instance~\cite[Lemma II.5.1]{Ga94}.
We will in the sequel always consider the quotient space topology
without change of notation.
In case that $\Omega$ is locally Lipschitz, it is known that
$\dot H^1_p(\Omega)\subset H^1_{p,\text{loc}}(\overline\Omega)$,
see \cite[Remark II.5.1]{Ga94}, and consequently,
any function in  $\dot H^1_p(\Omega)$ has a well-defined trace on
$\partial\Omega$.

For $s\in\R$ and $1<p<\infty$ we
also consider the {homogeneous Bessel-potential spaces} $\dot H^s_p(\R^n)$
of order $s$, defined by
\begin{equation}
\label{dot-H-s}
\begin{split}
&\dot H^s_p(\R^n):=(\{g\in {\mathcal S}^\prime(\R^n): \dot I^sg\in L_p(\R^n)\},
\|\cdot\|_{\dot H^s_p(\R^n)}),\\
& \|g\|_{\dot H^s_p(\R^n)}:=\|\dot I^sg\|_{L_p(\R^n)},
\end{split}
\end{equation}
where ${\mathcal S^\prime}(\R^n)$
denotes the space of all tempered distributions,
and $\dot I^s$ is the Riesz potential given by
$$
\dot I^s g:=(-\Delta)^{s/2}g:=
\mathcal F^{-1}(|\xi|^s \mathcal Fg),\quad g\in\mathcal S^\prime(\R^n).
$$
By factoring out all polynomials,
$\dot H^s_p(\R^n)$
becomes a Banach space with the natural quotient norm.
For $s\in\R\setminus\ZZ$,
the homogeneous Sobolev-Slobodeckij spaces $\dot W^s_p(\R^n)$
of fractional order
can be obtained  by real interpolation as
\begin{equation*}
\dot W^s_p(\R^n):=(\dot H^k_p(\R^n),\dot H^{k+1}_p(\R^n))_{s-k,p},
\quad k<s<k+1,
\end{equation*}
where $(\cdot,\cdot)_{\theta,p}$ is the real interpolation method.
It follows that
\begin{equation}
\label{I-s-isom}
\dot I^s\in
\text{Isom}(\dot H^{t+s}_p(\R^n),\dot H^t_p(\R^n))\cap
\text{Isom}(\dot W^{t+s}_p(\R^n),\dot W^t_p(\R^n)),
\quad s,t\in\R,
\end{equation}
with $\dot W^k_p=\dot H^k_p$ for $k\in\ZZ$.
We refer to \cite[Section 6.3]{BL76} and \cite[Section 5]{Tr83}
for more information on homogeneous functions spaces.
In particular, it follows from
parts (ii) and (iii) in \cite[Theorem 5.2.3.1]{Tr83}
that the definitions \eqref{dot-H-1} and \eqref{dot-H-s}
are consistent if $\Omega=\R^n$, $s=1$, and $1<p<\infty$.
We note in passing that
\begin{equation}
\label{norms-homogeneous}
          \left(\int_{\R^n}\int_{\R^n}
          \frac{|g(x)-g(y)|^p}
          {|x-y|^{n+sp}}\, dx\, dy\right)^{\!1/p}\!\!,
\ \
\left(\int_0^\infty t^{(1-s)p}\|\frac{d}{dt}P(t)g\|^p_{L_p(\R^n)}\frac{dt}{t}
\right)^{\!1/p}\!\!
\end{equation}
define equivalent norms on $\dot W^s_p(\R^n)$ for $0<s<1$,
where $P(\cdot)$ denotes the Poisson semigroup,
see \cite[Theorem 5.2.3.2 and Remark 5.2.3.4]{Tr83}.
Moreover,
\begin{equation}
\label{trace-homogeneous}
\gamma_{\pm}\in {\Li}(\dot W^1_p(\R^{n+1}_\pm), \dot W^{1-1/p}_p(\R^n)),
\end{equation}
where $\gamma_{\pm}$ denotes the trace operators,
see for instance \cite[Theorem II.8.2]{Ga94}.

\section{The Linearized Two-Phase Stokes Problem}
In this section we consider the linear two-phase
(inhomogeneous) Stokes problem
\begin{equation}
\label{linStokes}
\left\{
\begin{aligned}
\rho\partial_tu -\mu\Delta u+\nabla \pi&=f
    &\ \hbox{in}\quad &\dot\R^{n+1}\\
{\rm div}\,u&= f_d&\ \hbox{in}\quad &\dot\R^{n+1}\\
-[\![\mu\partial_y v]\!] -[\![\mu\nabla_{x}w]\!]&=g_v
    &\ \hbox{on}\quad &\R^n\\
-2[\![\mu\partial_y w]\!] +[\![\pi]\!] &=g_w
    &\ \hbox{on}\quad &\R^n\\
[\![u]\!] &=0 &\ \hbox{on}\quad &\R^n\\
u(0)&=u_0 &\ \hbox{in}\quad &\dot\R^{n+1}.\\
\end{aligned}
\right.
\end{equation}
Here the initial value $u_0$ as well as the inhomogeneities
$(f,f_d, g_v,g_w)$ are given. We want to establish
maximal regularity for this problem in the framework of $L_p$-spaces.
Thus we are interested in solutions
$(u,\pi)$ in the class
$$ 
u\in H^1_p(J;L_p(\R^{n+1},\R^{n+1}))
 \cap L_p(J;H^2_p(\dot{\R}^{n+1},\R^{n+1})),
\quad \pi\in L_p(J;\dot{H}^1_p(\dot{\R}^{n+1})).
$$
We remind here that  $J=[0,a]$ and
$\dot{\R}^{n+1}=\{(x,y)\in\R^{n}\times\R:\, y\neq0\}$.
If $(u,\pi)$ is a solution of (\ref{linStokes}) in this class
we necessarily have $f\in L_p(J;L_p(\R^{n+1}))$,
and additionally $u_0\in W^{2-2/p}_p(\dot{\R}^{n+1},\R^{n+1})$ by trace theory.
Moreover,
$$ 
f_d\in H^1_p(J;\dot{H}^{-1}_p(\R^{n+1}))
  \cap L_p(J;H^1_p(\dot{\R}^{n+1})),
$$   
as the operator ${\rm div}$ maps $L_p(\R^{n+1})$
onto $\dot{H}^{-1}_p(\R^{n+1})$.
Taking  traces at the interface $y=0$ results in
$ g_v\in W^{1/2-1/2p}_p(J;L_p(\R^{n},\R^n))
\cap L_p(J;W^{1-1/2p}_p(\R^{n},\R^n)),
$
and
$g_w\in L_p(J;\dot{W}^{1-1/p}_p(\R^n)).$
If, in addition,
$$[\![\pi]\!]\in W^{1/2-1/2p}_p(J;L_p(\R^{n}))
\cap L_p(J;W^{1-1/2p}_p(\R^{n}))
$$ then $g_w$
shares this regularity.

The main result  of this section states the converse of these assertions, i.e.\ maximal $L_p$-regularity for \eqref{linStokes}.
\begin{theorem}
\label{th:3.1}
Let $1<p<\infty$ be fixed, $p\neq 3/2,3$, and assume that $\rho_j$ and $\mu_j$ are positive
constants for $j=1,2$, and set $J=[0,a]$.
Then the  Stokes problem \eqref{linStokes} admits a unique solution $(u,\pi)$ with regularity
 $$u\in H^1_p(J;L_p(\R^{n+1},\R^{n+1}))
\cap L_p(J;H^2_p(\dot{\R}^{n+1},\R^{n+1})),
\quad \pi\in L_p(J;\dot{H}^1_p(\dot{\R}^{n+1})),$$
if and only if the data
$(f,f_d,g_v,g_w,u_0)$
satisfy the following regularity and compatibility conditions:
\begin{itemize}
\item[(a)]
$f\in L_p(J;L_p(\R^{n+1},\R^{n+1}))$,
\vspace{1mm}
\item[(b)] 
$f_d\in H^1_p(J; \dot{H}^{-1}_p(\R^{n+1}))\cap L_p(J; H^1_p(\dot{\R}^{n+1}))$,
\vspace{1mm}
\item[(c)]
$g_v\in W^{1/2-1/2p}_p(J;L_p(\R^{n},\R^n))\cap L_p(J;W^{1-1/p}_p(\R^{n},\R^n))$,\\
$g_w\in L_p(J;\dot W^{1-1/p}_p(\R^n))$,
\vspace{1mm}
\item[(d)]
$u_0\in W^{2-2/p}_p(\dot{\R}^{n+1},\R^{n+1})$,
\vspace{1mm}
\item[(e)]
${\rm div}\, u_0=f_d(0)$ in $\,\dot\R^{n+1}$ and $[\![u_0]\!]=0$
on $\,\R^n$ if $p>3/2$,
\vspace{1mm}
\item[(f)]
$-[\![\mu\partial_y v_0]\!] -[\![\mu\nabla_{x}w_0]\!] ={g_v}(0)$ on
$\,\R^n$ if $p>3$.
\vspace{1mm}
\end{itemize}
In addition,
$[\![\pi]\!]\in W^{1/2-1/2p}_p(J;L_p(\R^{n}))\cap L_p(J;W^{1-1/p}_p(\R^{n}))$
if and only if
$$ g_w\in W^{1/2-1/2p}_p(J;L_p(\R^{n}))\cap L_p(J;W^{1-1/p}_p(\R^{n})).$$
The solution map $[(f,f_d,g_v,g_w,g_h, u_0,h_0)\mapsto (u,\pi)]$ is continuous between the corresponding spaces.
\end{theorem}
\begin{proof}
The basic idea of the proof is to reduce system~\eqref{linStokes}
to the case where $(f,f_d,u_0)=(0,0,0)$ and $g_v(0)=0$,
and then to solve the resulting
problem by means of the {\em Dirichlet-to-Neumann operator} for the Stokes problem.
We can achieve this goal in four steps, as follows.
\smallskip\\
{\bf Step 1.}\, For given data $(f,g_v,u_0)$
subject to the conditions of the theorem we first solve the parabolic problem
without pressure and divergence, i.e. we solve
\begin{equation}
\label{parabolic}
\left\{
\begin{aligned}
\rho\partial_tu -\mu\Delta u&=f
    &\ \hbox{in}\quad &\dot\R^{n+1}\\
-[\![\mu\partial_y v]\!] -[\![\mu\nabla_{x}w]\!]&=g_v
    &\ \hbox{on}\quad &\R^n\\
-2[\![\mu\partial_y w]\!] &=\tilde g_w
    &\ \hbox{on}\quad &\R^n\\
[\![u]\!] &=0 &\ \hbox{on}\quad &\R^n\\
u(0)&=u_0 &\ \hbox{in}\quad &\dot\R^{n+1}.
\end{aligned}
\right.
\end{equation}
Here we set $\tilde{g}_w=-2e^{-D_nt}[\![\mu\partial_yw_0]\!]$ with $D_n:=-\Delta$ in $L_p(\R^n)$. The function~$\tilde g_w$ has the same regularity as $g_v$, and
the necessary compatibility conditions are satisfied.
By reflection of the $\{y<0\}$-part of this problem to the upper halfplane, we obtain a parabolic system on a halfspace with
boundary conditions satisfying the Lopatinskii-Shapiro conditions.
Therefore, the theory of parabolic boundary value problems yields
a unique solution $u_1$ for \eqref{parabolic}
with regularity 
$$u_1\in H^1_p(J;L_p(\R^{n+1},\R^{n+1}))
\cap L_p(J;H^2_p(\dot{\R}^{n+1},\R^{n+1})).$$
We refer to
Denk, Hieber and Pr\"uss \cite{DHP03,DHP07} for this.
\smallskip\\
{\bf Step 2.}
In this step we solve the  Stokes equations
\begin{equation}
\label{Navier-Stokes}
\left\{
\begin{aligned}
\rho\partial_tu -\mu\Delta u+\nabla \pi&=0
    &\ \hbox{in}\quad &\dot\R^{n+1}\\
{\rm div}\,u&= f_d-{\rm div}\,u_1&\ \hbox{in}\quad &\dot\R^{n+1}\\
u(0)&=0 &\ \hbox{in}\quad &\dot\R^{n+1}\\
\end{aligned}
\right.
\end{equation}
where $u_1$ is the solution obtained in Step 1.
It follows from assumption (e) that system \eqref{Navier-Stokes}
satisfies the  compatibility condition
${\rm div}\,u(0)=f_d(0)-{\rm div}\,u_1(0)=0$.
We remind that $\rho=\rho_2\chi_{\R^{n+1}_+}+\rho_1\chi_{\R^{n+1}_-}$
and $\mu=\mu_2\chi_{\R^{n+1}_+}+\mu_1\chi_{\R^{n+1}_-}$.
Concentrating on the upper halfplane,
we extend the function $(f_d-{\rm div}\,u_1)$ evenly in $y$ to all of $\R^{n+1}$ and
solve the Stokes problem with coefficients $\rho_2,\mu_2$ in the whole space,
see \cite[Theorem 5.1]{BP07}.
This gives a solution which has the property that the normal
velocity $w$ vanishes at the interface; the latter is due to the symmetries of the equations.
We restrict this solution to $\R^{n+1}_+$.
We then do the same on the lower halfplane.
This results in a solution $(u_2,\pi_2)$ for system
\eqref{Navier-Stokes} that satisfies
\begin{equation*}
\begin{split} 
&u_2\in H^1_p(J;L_p(\R^{n+1},\R^{n+1}))
\cap L_p(J;H^2_p(\dot{\R}^{n+1},\R^{n+1})), \\
&\pi_2\in L_p(J;\dot{H}^1_p(\dot{\R}^{n+1})),
\quad w_2=0 \text{ on $\R^n$},
\end{split}
\end{equation*}
where, as before, $u_2=(v_2,w_2)$.
We remark that the tangential part of the velocity,
i.e. $v_2$, may now have a jump at the boundary $y=0$.
\smallskip\\
{\bf Step 3.}
To remove the jump in the tangential velocity,
we solve the homogeneous Stokes problem in the lower halfplane
with this jump as Dirichlet datum, that is, we solve
\begin{equation}
\label{tangential-jump}
\left\{
\begin{aligned}
\rho_1\partial_tu -\mu_1\Delta u+\nabla \pi&=0
    &\ \hbox{in}\quad &\dot\R^{n+1}_{-}\\
{\rm div}\,u&= 0 &\ \hbox{in}\quad &\dot\R^{n+1}_{-}\\
v=[\![v_2]\!],\ w&=0   &\ \hbox{on}\quad &\R^{n} \\
u_0(0)&=0 &\ \hbox{in}\quad &\dot\R^{n+1}_{-}\\
\end{aligned}
\right.
\end{equation}
where $u_2=(v_2,w_2)$ is the solution obtained in Step 2.
It follows from Proposition~\ref{pro:3.3} below that system \eqref{tangential-jump}
has a unique solution with the regularity properties
of Theorem~\ref{th:3.1}.
Let $(u_3,\pi_3)$ be defined by
\begin{equation*}
(u_3,\pi_3):=
\left\{
\begin{aligned}
&\text{(0,0)} &\ \hbox{in} \quad &\dot\R^{n+1}_{+} \\
&\text{the solution of \eqref{tangential-jump}}
&\ \hbox{in} \quad &\dot\R^{n+1}_{-}. \\
\end{aligned}
\right.
\end{equation*}
Then $(u_3,\pi_3)$ also satisfies the regularity properties
stated in Theorem~\ref{th:3.1}
and we have
$[\![v_3]\!]=-[\![v_2]\!]$ and $[\![w_3]\!]=0$ on $\R^n$.
\smallskip\\
{\bf Step 4.}
In this step we consider the problem
\begin{equation}
\label{g}
\left\{
\begin{aligned}
\rho\partial_tu -\mu\Delta u+\nabla \pi&=0
    &\hbox{in}\ &\dot\R^{n+1}\\
{\rm div}\,u&= 0 &\hbox{in}\ &\dot\R^{n+1}\\
-[\![\mu\partial_y v]\!] -[\![\mu\nabla_x w]\!]
&=[\![\mu\partial_y(v_2+v_3)]\!]+[\![\mu\nabla_x(w_2+w_3)]\!]
    &\hbox{on}\ &\R^n\\
-2[\![\mu\partial_y w]\!] +[\![\pi]\!] &
=g_w-\tilde g_w+ 2[\![\mu\partial_y(w_2+w_3)]\!]-[\![\pi_2+\pi_3]\!]
    &\hbox{on}\ &\R^n\\
[\![u]\!] &=0 &\hbox{on}\ &\R^n\\
u(0)&=0 &\ \hbox{in}\ &\dot\R^{n+1}\\
\end{aligned}
\right.
\end{equation}
with $(v_2,w_2,\pi_2)$ and $(v_3,w_3,\pi_3)$ the solutions
obtained in Steps 2 and 3.
Here it should be observed that
the function on the right hand side of line 3 appearing
as boundary condition has zero time trace.
Problem \eqref{g}, which is also of independent interest,
will be studied in detail in the next section.
It will be shown in  Corollary~\ref{co:4.2} that it admits
a unique solution, denoted here by $(u_4,\pi_4)$, which satisfies the regularity properties
stated in Theorem~\ref{th:3.1}.\smallskip\\
To finish the proof of Theorem~\ref{th:3.1} we set
$(u,\pi)=(\sum_{i=1}^4u_i,\sum_{i=1}^4\pi_i)$,
where  $(u_i,\pi_i)$ are the solutions obtained
in Step i, with $\pi_1:=0$.
Then $(u,\pi)$ satisfies the regularity properties
stated in the Theorem and it is the unique solution of
\eqref{linStokes}.
\end{proof}
\begin{remark}
We refer to the recent paper by Bothe and Pr\"uss \cite{BP07}
for results related to Theorem 3.1 for the more general and involved
situation of a generalized Newtonian fluid.
\end{remark}
Let us now consider the problem
\begin{equation}
\label{SPD}
\left\{
\begin{aligned}
\rho\partial_tu -\mu\Delta u+\nabla \pi&= 0
    &\ \hbox{in}\quad &\dot\R^{n+1}\\
{\rm div}\,u&=0&\ \hbox{in}\quad &\dot\R^{n+1}\\
 u &=u_b &\ \hbox{on}\quad &\R^n\\
u(0)&=0 &\ \hbox{in}\quad &\dot\R^{n+1}\\
\end{aligned}
\right.
\end{equation}
and prove the result that was used in Step 3 above.
\begin{proposition}
\label{pro:3.3}
Let $1<p<\infty$ and assume that $\rho_j$ and $\mu_j$ are positive constants, $j=1,2$, and set $J=[0,a]$.
Then problem \eqref{SPD} admits a unique solution $(u,\pi)$ with
$$
u\in {_0}H^1_p(J;L_p(\R^{n+1},\R^{n+1}))
\cap L_p(J;H^2_p(\dot{\R}^{n+1},\R^{n+1})),
\quad \pi\in L_p(J;\dot{H}^1_p(\dot{\R}^{n+1}))
$$
if and only if the data $u_b=(v_{b},w_{b})$ satisfy the following regularity
assumptions
\begin{itemize}
\item[(a)] $ v_{b}\in {_0W}^{1-1/2p}_p(J;L_p(\R^{n},\R^n))
 \cap L_p(J;W^{2-1/p}_p(\R^{n},\R^n))$,
\vspace{1mm}
\item[(b)]
$w_{b} \in {_0H}^1(J;\dot{W}^{-1/p}_p(\R^n))\cap L_p(J;W^{2-1/p}_p(\R^n))$.\\
\end{itemize}
\end{proposition}
\begin{proof}
(i) Assume for a moment that we have a solution in the proper regularity class
even on the half-line $J=\R_+$.
Then we may employ the Laplace transform in $t$ and the Fourier transform
in the
tangential variables $x\in\R^{n}$, to obtain the following boundary value problem
for a system of ordinary differential equations on~$\dot{\R}:$
\begin{equation*}
\label{ODE}
\left\{
\begin{aligned}
\omega^2\hat v-\mu\partial_y^2\hat v +i\xi\hat \pi&=0, &\quad y\neq 0, \\
\omega^2\hat w-\mu\partial_y^2\hat w +\partial_y\hat \pi&=0, &\quad y\neq 0, \\
    (i\xi|\hat v)+\partial_y\hat w &=0,&\quad y\neq 0,\\
\hat v(0)=\hat v_b,\ \hat w(0)&=\hat w_b. \\
\end{aligned}
\right.
\end{equation*}
Here we have set $\omega^2_j=\rho_j\lambda+\mu_j|\xi|^2$, $j=1,2$,
and
$$\hat v_j(\lambda,\xi,y)=(2\pi)^{-n/2}\int_0^\infty\int_{\R^n}
e^{-\lambda t}e^{-i(x|\xi)}v(t,x,y)\,dx\,dt,
\quad (-1)^jy>0.
$$
This system of equations is easily solved to the result
\begin{equation}
\label{v_2}
\left[\begin{array}{c}\hat v_2\\\hat w_2\\ \hat \pi_2\end{array}\right]
= e^{-\omega_2y/\sqrt{\mu_2}}\left[\begin{array}{c} a_2\\
\frac{\sqrt{\mu_2}}{\omega_2}(i\xi|a_2)\\ 0\end{array} \right] + \alpha_2 e^{-|\xi|y} \left[\begin{array}{c} -i\xi\\ |\xi|\\ \rho_2\lambda
\end{array}\right],
\end{equation}
for $y>0$, and
\begin{equation}
\label{v_1}
\left[\begin{array}{c}\hat v_1\\ \hat w_1\\ \hat \pi_1\end{array}\right]= e^{\omega_1y/\sqrt{\mu_1}}\left[\begin{array}{c} a_1\\
-\frac{\sqrt{\mu_1}}{\omega_1}(i\xi|a_1)\\ 0\end{array} \right] + \alpha_1 e^{|\xi|y} \left[\begin{array}{c} -i\xi\\ -|\xi|\\ \rho_1\lambda
\end{array}\right],
\end{equation}
for $y<0$.
Here $a_i\in\R^{n}$ and $\alpha_i$ have to be determined by the boundary conditions
$\hat v(0)=\hat v_{b}$ and $\hat w(0)=\hat w_{b}$. We have
$$ a_2 -i\xi\alpha_2 = \hat v_{b}= a_1-i\xi\alpha_1,$$
and
$$
\frac{\sqrt{\mu_2}}{\omega_2}(i\xi|a_2)+|\xi|\alpha_2
=\hat w_{b} = -\frac{\sqrt{\mu_1}}{\omega_1}(i\xi|a_1)-|\xi|\alpha_1
$$
where $(a|b):=\sum a^jb^j$ for $a,b\in \C^n$.
This yields
\begin{equation}
\label{a_j-alpha_j}
\begin{split}
&a_j=\hat v_{b}+ i\xi\alpha_j ,\quad j=1,2, \\
&\alpha_2 = -\frac{\omega_2+\sqrt{\mu_2}|\xi|}{\rho_2\lambda|\xi|}
(\sqrt{\mu_2}(i\xi|\hat v_{b})-\omega_2 \hat w_{b}),\\
&\alpha_1 = -\frac{\omega_1+\sqrt{\mu_1}|\xi|}{\rho_1\lambda|\xi|}
(\sqrt{\mu_1}(i\xi|\hat v_{b})+\omega_1 \hat w_{b}).
\end{split}
\end{equation}
\smallskip\\
(ii)
By parabolic theory, the velocity $u$ has the correct regularity
provided the pressure gradient is in $L_p$, and provided
$$
u_b\in {_0W}^{1-1/2p}_p(J;L_p(\R^n,\R^{n+1}))
\cap L_p(J;W^{2-1/p}_p(\R^n,\R^{n+1})),
$$
see for instance Denk, Hieber and Pr\"uss \cite{DHP07}.
In particular this regularity of $u_b$ is necessary. Note that the embedding
\begin{equation}
\label{embedding-negative}
{_0H}^1_p(J;\dot{W}^{-1/p}_p(\R^n))\cap L_p(J;W^{2-1/p}_p(\R^n))
\hookrightarrow {_0W}^{1-1/2p}_p(J;L_p(\R^n))
\end{equation}
is valid. 
This follows from the fact
that $\dot W^{-1/p}_p(\R^n)\hookrightarrow W^{-1/p}_p(\R^n)$
by a similar argument as in the proof of \cite[Lemma 6.3]{PSS07}
where we set $Au:=(1-\Delta)u$.
\smallskip\\
(iii)
We will now introduce some operators that will
play a crucial role in our analysis.
We set $G:=\partial_t$ in $X:=L_p(J;L_p(\R^n))$ with domain
$$
{\sf D}(G)={_0H}^1_p(J;L_p(\R^n)).$$
Then it is well-known
that $G$ is closed, invertible and sectorial with angle $\pi/2$, and $-G$ is the generator
of a $C_0$-semigroup of contractions in $L_p(\R^n)$. Moreover, $G$ admits an
$H^\infty$-calculus in $X$ with $H^\infty$-angle $\pi/2$ as well; see e.g.\
\cite{HiPr98}. The symbol of $G$ is $\lambda$, the time covariable.

Next we set $D_n:=-\Delta$, the Laplacian in $L_p(\R^n)$ with domain
${\sf D}(D_n)=H^2_p(\R^n)$. It is also well-known that $D_n$ is closed and  sectorial with
angle $0$, and it admits a bounded $H^\infty$-calculus which is even $\cR$-bounded with
$\cR H^\infty$-angle $0$; see e.g.\  \cite{DHP01a}.
 These results also hold for the canonical extension of
$D_n$ to $X$, and also for the fractional power $D_n^{1/2}$ of $D_n$.
Note that the domain of $D^{1/2}_n$ is
$${\sf D}(D_n^{1/2})=L_p(J;H^1_p(\R^n)).
$$
The symbol of
$D_n$ is $|\xi|^2$, that of $D_n^{1/2}$ is given by $|\xi|$, where $\xi$ means the covariable
of $x$.
By the Dore-Venni theorem for sums of commuting sectorial operators, cf. \cite{DoVe87,PrSo90}, we see that the parabolic operators $L_j:=\rho_j G+\mu_j D_n$ with
natural domain
$${\sf D}(L_j)={\sf D}(G)\cap{\sf D}(D_n)={_0H}^{1}_p(J;L_p(\R^n))\cap L_p(J;H^2_p(\R^n))$$
are closed, invertible and  sectorial with angle $\pi/2$.
Moreover, $L_j$ also admits a bounded
$H^\infty$-calculus in $X$ with $H^\infty$-angle $\pi/2$; cf. e.g.\
\cite{PrSi06}. The same results are valid for
the operators $F_j=L_j^{1/2}$, their $H^\infty$-angle is $\pi/4$, and their domains are
$$
{\sf D}(F_j)= {\sf D}(G^{1/2})\cap {\sf D}(D_n^{1/2})={_0H}^{1/2}_p(J;L_p(\R^n))\cap
L_p(J;H^1_p(\R^n)).
$$
The symbol of $L_j$ is $\rho_j\lambda+\mu_j|\xi|^2$ and that of $F_j$
is given by $\sqrt{\rho_j\lambda+\mu_j|\xi|^2}$.
\smallskip\\
Let $R$ denote the Riesz operator with symbol $\zeta=\xi/|\xi|$.
It follows from the Mikhlin-H\"ormander theorem that
$R$ is a bounded linear operator on $W^s_p(\R^n)$, and hence also
on $L_p(J;W^s_p(\R^n)$ by canonical extension.

(iv)
Let $\beta_2=\rho_2\lambda\alpha_2$.
Then the transform of the pressure $\pi_2$ in $\R^{n+1}_{+}$ is given by
$e^{-|\xi|y}\beta_2$.
The pressure gradient will be in $L_p$ provided the inverse transform of $\beta_2$
is in the space $L_p(J;\dot{W}_p^{1-1/p}(\R^n))$.
In fact, $e^{-|\xi|y}$ is the symbol of
the Poisson semigroup $P(\cdot)$ in $L_p(\R^n)$, and
the negative generator of $P(\cdot)$ is $D^{1/2}_n$.
Then  the second part of ~\eqref{norms-homogeneous} shows that
$D^{1/2}_nP(\cdot)\beta_2\in L_p(\R_+;L_p(\R^n))$ if and only
$\beta_2\in \dot{W}_p^{1-1/p}(\R^n)$.
This result extends canonically to $L_p(J;L_p(\R^{n+1}_+))$.
\medskip\\
Therefore, let us look more closely at $\beta_2$. We easily obtain
$$
\beta_2= \rho_2\frac{\lambda}{|\xi|}\hat w_{b}
+(\sqrt{\mu}_2\omega_2+\mu_2|\xi|)(\hat w_{b}-(i\zeta|\hat v_{b})),
$$
where $\zeta=\xi/|\xi|$.
We recall that $\dot{D}_n^{1/2}:=\mathcal F^{-1}(|\xi|\mathcal F\,\cdot):
\dot{W}^s_p(\R^n)\to\dot{W}^{s-1}_p(\R^n)$ is an isomorphism.

With the operators introduced above, $b_2$, the inverse transform
of $\beta_2$, can be represented by
$$b_2 = \rho_2 G\dot{D}^{-1/2}_nw_{b} + (\sqrt{\mu_2}F_2+\mu_2D^{1/2}_n)(w_{b}-i(R|v_{b}))
=:b_{21}+b_{22}.$$
Due to \eqref{embedding-negative} and
$
{_0W}^{1-1/2p}_p(J;L_p(\R^{n}))
\cap L_p(J;W^{2-1/p}_p(\R^{n}))=D_{F_j}(2-1/p,p),$
the second term $b_{22}$ is in
$$D_{F_j}(1-1/p,p)={_0W}^{1/2-1/2p}_p(J;L_p(\R^{n}))
\cap L_p(J;W^{1-1/p}_p(\R^{n})),$$
which embeds into $L_p(J;\dot{W}^{1-1/p}_p(\R^n)).$
Here we used the notation
$$
D_{F_j}(\theta,p)=(X,{\sf D}(F_j))_{\theta,p},
\quad D_{F_j}(1+\theta,p)=({\sf D}(F_j),{\sf D}(F^2_j))_{\theta,p},
\quad \theta\in (0,1).
$$
Thus it remains to look at the first term $b_{21}=\rho_2GD^{-1/2}_n w_b$.
Since
$$G\dot{D}_n^{-1/2}: {_0H}^1_p(J;\dot{W}^{-1/p}_p(\R^n))
\rightarrow L_p(J;\dot{W}_p^{1-1/p}(\R^n))$$
is bounded and invertible, we see that the condition
$w_b\in{_0H}^1_p(J;\dot{W}^{-1/p}_p(\R^n))$ is necessary and sufficient
for $b_{21}\in L_p(J;\dot{W}^{1-1/p}_p(\R^n))$.
Of course, similar arguments apply for the lower half-plane.
\end{proof}
\section{The Dirichlet-to-Neumann Operator for the Stokes Equation}
The main ingredient in analyzing problem \eqref{linStokes}
with $(f,f_d,u_0)=(0,0,0)$ and $g_v(0)=0$
is the Dirichlet-to-Neumann operator.
It is defined as follows.
Let $(u,\pi)$ be the solution of the Stokes problem \eqref{SPD}
with Dirichlet boundary condition $u_b$ on $\R^n$,
see Proposition~\ref{pro:3.3}.
We then define the Dirichlet-to-Neumann operator
by means of
\begin{equation}
\label{def-ND}
(\cDN) u_b=-[\![S(u,\pi)]\!]e_{n+1}
= -[\![\mu\big(\nabla u+(\nabla u)^{\sf T}\big)]\!]e_{n+1} +[\![\pi]\!]e_{n+1}.
\end{equation}
For this purpose it is convenient to split $u$ into $u=(v,w)$ as before,
and $u_b$ into $u_b=(v_{b},w_{b})$. Then we obtain
\begin{equation}
\label{def-ND-2}
(\cDN) u_b= (-[\![\mu\partial_y v]\!]-[\![\mu\nabla_x w]\!],
-2[\![\mu\partial_yw]\!]+ [\![\pi]\!]) .
\end{equation}
We will now formulate and prove the main result of this section.
\begin{theorem}
\label{th:4.1}
The Dirichlet-to-Neumann operator $\cDN$ for the Stokes problem
is an isomorphism from the {\em Dirichlet space} $u_b=(v_b,w_b)$ with
\begin{equation*}
\begin{aligned}
& v_b\in {_0W}^{1-1/2p}_p(J;L_p(\R^n,\R^n))\cap L_p(J;W^{2-1/p}_p(\R^n,\R^n)),\\
& w_b\in {_0H}^{1}_p(J;\dot{W}^{-1/p}_p(\R^n))\cap L_p(J;W^{2-1/p}_p(\R^n))
\end{aligned}
\end{equation*} 
onto the {\em Neumann space} $g=(g_v,g_w)$ with
\begin{equation*}
\begin{aligned}
&g_v\in {_0W}^{1/2-1/2p}_p(J;L_p(\R^n,\R^n))\cap L_p(J;W^{1-1/p}_p(\R^n,\R^n)),\\
&g_w\in  L_p(J;\dot{W}^{1-1/p}_p(\R^n)).
\end{aligned}
\end{equation*}
\end{theorem}
\begin{proof}
(i) Let $(\hat v_1,\hat w_1,\hat \pi_1)$ and
$(\hat v_2,\hat w_2,\hat \pi_2)$
be as in \eqref{v_2}--\eqref{v_1}.
We may now compute the symbol of the Dirichlet-to-Neumann operator to the result
$$ (\cDN) \hat u_b= \left[\begin{array}{c} \omega_1\sqrt{\mu_1}a_1+\omega_2\sqrt{\mu_2}a_2-(\alpha_1\mu_1+\alpha_2\mu_2)|\xi|i\xi
- [\![\mu]\!]i\xi\hat w_{b}\\
 2i(\mu_2 a_2-\mu_1 a_1|\xi) +2(\alpha_2\mu_2-\alpha_1\mu_1)|\xi|^2+\lambda(\alpha_2\rho_2-\alpha_1\rho_1)\end{array}\right]$$
where the functions $\alpha_j$ and $a_j$ are given in \eqref{a_j-alpha_j}.
Simple  algebraic manipulations then yield
the following symbol
\begin{equation}\label{DN}
(\cDN)(\lambda,\xi)=\left[\begin{array}{cc}
\alpha +\beta\zeta\otimes\zeta & i\gamma\zeta\\
-i\gamma\zeta^T& \alpha+\delta\end{array}\right],
\end{equation}
where $\zeta=\xi/|\xi|$ and
\begin{equation}
\label{alpha-beta}
\begin{split}
&\alpha =\sqrt{\mu_1}\omega_1+\sqrt{\mu_2}\omega_2,
\quad\beta =(\mu_1+\mu_2)|\xi|, \\
&\gamma = (\sqrt{\mu_2}\omega_2-\sqrt{\mu_1}\omega_1)-[\![\mu]\!]|\xi|,
\quad \delta = (\omega_1^2+\omega_2^2)/|\xi|=\beta +(\rho_1+\rho_2)\lambda/|\xi|.\\
\end{split}
\end{equation}
\noindent
\smallskip\\
Next we want to compute the inverse of the Dirichlet-to-Neumann operator.
Thus we have to solve the equation
$(\cDN) u_b = g$. As before we use the decomposition $u_b=(v_b,w_b)$ and $g=(g_v,g_w)$. Then in transformed variables we have the system
\begin{equation*}
\begin{split}
\alpha \hat v_b + \beta \zeta (\zeta |\hat v_b)
+i\gamma\zeta \hat w_b &=\hat g_v,\\
-i\gamma(\zeta |\hat v_b) +(\alpha+\delta)\hat w_b &=\hat g_w.
\end{split}
\end{equation*}
This yields
\begin{equation}
\label{DN-1-v}
\hat v_b= \alpha^{-1}
[\hat g_v -\zeta(\beta(\zeta |\hat v_b)+i\gamma \hat w_b)].
\end{equation}
(ii)
This last equation shows that it is sufficient
to determine $(\hat v_b|\zeta)$ and $\hat w_b$.
If the inverses of $\beta(\hat v_b|\zeta)$ and $\gamma \hat w_b$ belong to the class
of $g_v$, then $v_b$ is uniquely determined and has the claimed regularity.
Indeed, $\alpha$ is the symbol of
$$
F:=\sqrt{\mu_1}F_1 + \sqrt{\mu_2} F_2,
\qquad {\sf D}(F)={_0H}^{1/2}_p(J;L_p(\R^n))\cap L_p(J;H^1_p(\R^n)),
$$
which is a bounded invertible operator from
its domain into $L_p(J;L_p(\R^n))$,
and hence also from $D_F(2-1/p,p)$ into $D_F(1-1/p,p).$
Here we note that
$$
D_F(\theta,p)=D_{F_j}(\theta,p)={_0W}^{\theta/2}_p(J;L_p(\R^n))
\cap L_p(J;W^\theta_p(\R^n)),
$$
for $\theta\in (0,2)$, $\theta\neq 1$.
Therefore, $F^{-1}g_v$ belongs to $D_F(2-1/p,p)$
if and only if $g_v\in D_F(1-1/p,p)$.
Next we note that $\gamma$ is the symbol of
$
\sqrt{\mu_2} F_2 -\sqrt{\mu_1} F_1 -[\![\mu]\!]D_n^{1/2}
$
which is bounded from $D_F(2-1/p,p)$ to $D_F(1-1/p,p)$, and
$\beta$ is the symbol of $(\mu_1+\mu_2)D_n^{1/2}$ which has the same mapping properties.
\smallskip\\
(iii)
It remains to show that $w_b$ and
$(R|v_b)$ belong to $D_F(2-1/p,p)$.
For $\hat w_b$ and $(\zeta|\hat v_b)$ we have the equations
\begin{equation*}
\begin{split}
(\alpha  + \beta) (\zeta|\hat v_b) +i\gamma \hat w_b
&=(\zeta|\hat g_v),\\
-i\gamma(\zeta|\hat v_b) +(\alpha+\delta)\hat w_b &=\hat g_w
\end{split}
\end{equation*}
since $|\zeta|=1$. Solving this 2-D system we obtain
\begin{equation}
\label{v-w-1}
\begin{aligned}
\hat w_b &= m^{-1}[ i\gamma (\zeta |\hat g_v) +(\alpha+\beta)\hat g_w], \\
(\zeta |\hat v_b) &= m^{-1}[ (\alpha+\delta)(\zeta |\hat g_v)-i\gamma\hat g_w],
\end{aligned}
\end{equation}
where
\begin{equation*}
\label{m}
m=(\alpha+\beta)(\alpha+\delta)-\gamma^2.
\end{equation*}
Since $\delta =\beta +(\rho_1+\rho_2)\lambda/|\xi|$ we obtain the following relation for $m$
\begin{equation*}
m = (\alpha+\beta)[(\rho_1+\rho_2)\frac{\lambda}{|\xi|} + 4 ( \frac{1}{\eta_1}+\frac{1}{\eta_2})^{-1}]=:(\alpha+\beta)n,
\end{equation*}
where $\eta_1= \sqrt{\mu_1}\omega_1 + \mu_2 |\xi|$ and $\eta_2= \sqrt{\mu_2}\omega_2 + \mu_1 |\xi|$.
This yields
\begin{equation}
\label{vw-symbol}
\begin{split}
\hat w_b &=  \frac{i\gamma}{(\alpha+\beta)n}(\zeta |\hat g_v)
+ \frac{\hat g_w}{n}, \\
 (\zeta |\hat v_b) &=
 \frac{(\rho_1+\rho_2)\lambda/|\xi|}{(\alpha+\beta)n}
 (\zeta |\hat g_v)+\frac{1}{n}
[(\zeta |\hat g_v)-\frac{i\gamma}{\alpha+\beta}\hat g_w].
\end{split}
\end{equation}
We define the operators $T_j$ by means of their symbols $\eta_j$, i.e.
$$
T_1:= \sqrt{\mu_1} F_1 + \mu_2 D_n^{1/2},
\quad T_2 :=\sqrt{\mu_2} F_2 +\mu_1 D_n^{1/2},
\quad {\sf D}(T_j)={\sf D}(F_j)={\sf D}(F)
$$
Then by the Dore-Venni theorem operators, the operators $T_j$
with domains $D(T_j)=D(F_j)=D(F)$
are invertible, sectorial with angle $\pi/4$. Moreover, they admit an $H^\infty$-calculus with
$H^\infty$-angle $\pi/4$, see for instance \cite{PrSi06}.
The harmonic mean $T$ of $T_1$ and $T_2$, i.e.
$$T:=2T_1T_2(T_1+T_2)^{-1}= 2(T^{-1}_1+T^{-1}_2)^{-1}$$
enjoys the same properties, as another application of the Dore-Venni theorem shows. The symbol
of $T$ is given by $\eta:=2\eta_1\eta_2/(\eta_1+\eta_2)$.
\smallskip\\
Next we consider the operator $GD_n^{-1/2}$ with domain
\begin{equation*}
\begin{split}
{\sf D}(GD_n^{-1/2})&=\{h\in \cR(D_n^{1/2}):\; D_n^{-1/2}h\in \cD(G)\}\\
&={_0H}^1_p(J;\dot{H}^{-1}_p(\R^n))\cap L_p(J;L_p(\R^n))
\end{split}
\end{equation*}
The inclusion from left to right in the last equality is obvious. The converse can be seen as follows.  
Let $h\in{_0H}^1_p(J;\dot{H}^{-1}_p(\R^n))\cap L_p(J;L_p(\R^n))$ and 
define $g:=\dot{D}_n^{-1/2}h$. Then 
$$g\in {_0H}^1_p(J;L_p(\R^n))\cap L_p(J;\dot{H}^1_p(\R^n))\hookrightarrow L_p(J;H^1_p(\R^n)),$$ 
and $D_n^{1/2}g =\dot{D}_n^{1/2}g=h\in L_p(J;L_p(\R^n))$,
which implies that $h\in \cR(D_n^{1/2})$ and $g=\dot{D}_n^{-1/2}h= D_n^{-1/2}h\in \cD(G).$
The operator
$GD_n^{-1/2}$ is closed, sectorial and admits a bounded $H^\infty$-calculus
with $H^\infty$-angle
$\pi/2$ on $X=L_p(J;L_p(\R^n))$; see for instance \cite[Corollary 2.2]{HaHi05}.
Its symbol is given by $\lambda/|\xi|$.
\smallskip\\
Finally, we consider the operator
\begin{equation}
\label{N}
N:=(\rho_1+\rho_2)GD_n^{-1/2} + 2 T,
\end{equation}
with domain
\begin{equation*}
{\sf D}(N)={\sf D}(GD_n^{-1/2})\cap {\sf D}(T)={_0H}^1_p(J;\dot{H}^{-1}_p(\R^n))\cap L_p(J;H^1_p(\R^n));
\quad 
\end{equation*}
recall \eqref{embedding-negative}.
By the Dore-Venni theorem $N$ is closed, invertible, and by \cite{PrSi06}
admits a bounded $H^\infty$-calculus as well, with $H^\infty$-angle $\pi/2$.
Its symbol is $n$.
\medskip\\
The operator with symbol $\gamma$ is then given by $T_2-T_1$,
and the operator with symbol $\alpha+\beta$ by $T_1+T_2$.
For the inverse transforms $w_b$ and $(R|v_b)$ of
$\hat w_b$ and $(\zeta |\hat v_b)$ we then obtain the representations
\begin{equation}
\begin{split}
\label{rep}
  w_b &= N^{-1}[(T_2-T_1)(T_1+T_2)^{-1}i(R|g_v) + g_w]\\
(R|v_b)&= (T_1+T_2)^{-1}(\rho_1+\rho_2)GD^{-1/2}_n N^{-1}(R|g_v) \\
&\qquad +  N^{-1}(R|g_v)-(T_2-T_1)(T_1+T_2)^{-1}N^{-1}ig_w\,.
\end{split}
\end{equation}
We note that $N^{-1}$ has the following mapping properties
\begin{equation*}
\begin{aligned}
& N^{-1}: L_p(J;L_p(\R^n))\rightarrow \ {_0H}^1_p(J;\dot{H}^{-1}_p(\R^n))\cap L_p(J;H^1_p(\R^n)) \hookrightarrow  L_p(J;L_p(\R^n)), \\
& N^{-1}: L_p(J;\dot{H}^1_p(\R^n))\rightarrow \ {_0H}^1_p(J;L_p(\R^n))\cap L_p(J;\dot{H}^2_p(\R^n))\hookrightarrow  L_p(J;L_p(\R^n)).
\end{aligned}
\end{equation*}
Therefore by three-fold real interpolation
\begin{equation}
\label{N-inverse-Lp}
N^{-1}: L_p(J;\dot{W}^{1-1/p}_p(\R^n))\rightarrow {_0H}^1_p(J;\dot{W}^{-1/p}_p(\R^n))\cap L_p(J;W^{2-1/p}_p(\R^n)).
\end{equation}
Moreover, $N^{-1}$ maps ${_0W}^{1/2-1/2p}_p(J;L_p(\R^n))$ into
\begin{equation}
\label{N-inverse-W^s}
{_0W}^{3/2-1/2p}_p(J;\dot H^{-1}_p(\R^n))\cap{_0W}^{1/2-1/2p}_p(J;H^1_p(\R^n)).
\end{equation}
Next we note that
the operators $T_j(T_1+T_2)^{-1}$ are bounded in $D_F(1-1/p,p)$,
as is the Riesz transform $R$,
and the assertion for $w_b$ follows now from~\eqref{rep}--\eqref{N-inverse-Lp}
and
\begin{equation*}
{_0W}^{1/2-1/2p}_p(J;L_p(\R^n))\cap L_p(J;W^{1-1/p}_p(\R^n))
\hookrightarrow  L_p(J;\dot{W}^{1-1/p}_p(\R^n))
\end{equation*}
The assertions for $(R|v_b)$ follow readily from
\eqref{embedding-negative} and ~\eqref{rep}--\eqref{N-inverse-W^s}.
\end{proof}
We can now formulate our second main result of this section
concerning the solvability of the  problem
\begin{equation}
\label{linStokes-homogen}
\left\{
\begin{aligned}
\rho\partial_tu -\mu\Delta u+\nabla \pi&=0
    &\ \hbox{in}\quad &\dot\R^{n+1}\\
{\rm div}\,u&=0&\ \hbox{in}\quad &\dot\R^{n+1}\\
-[\![\mu\partial_y v]\!] -[\![\mu\nabla_{x}w]\!]&=g_v
    &\ \hbox{on}\quad &\R^n\\
-2[\![\mu\partial_y w]\!] +[\![\pi]\!] &=g_w
    &\ \hbox{on}\quad &\R^n\\
[\![u]\!] &=0 &\ \hbox{on}\quad &\R^n\\
u(0)&=0 &\ \hbox{in}\quad &\dot\R^{n+1}.\\
\end{aligned}
\right.
\end{equation}
\begin{corollary}
\label{co:4.2}
Let $1<p<\infty$ and assume
that $\rho_j$ and $\mu_j$ are positive constants, $j=1,2$, and set $J=[0,a]$.
Then \eqref{linStokes-homogen}
admits a unique solution $(u,\pi)$ with
$$
u\in {_0H}^1_p(J;L_p(\R^{n+1},\R^{n+1}))
      \cap L_p(J;H^2_p(\dot{\R}^{n+1},\R^{n+1})),
\quad \pi\in L_p(J;\dot{H}^1_p(\dot{\R}^{n+1}))
$$
if and only if $g=(g_{v},g_{w})$ satisfies the following regularity
assumptions
\begin{enumerate}
\item[(a)]
$g_v\in {_0W}^{1/2-1/2p}_p(J;L_p(\R^n,\R^n))\cap L_p(J;W^{1-1/p}_p(\R^n,\R^n)),$
\vspace{1mm}
\item[(b)]
$g_w\in  L_p(J;\dot{W}^{1-1/p}_p(\R^n)).$
\end{enumerate}
\end{corollary}
\begin{proof}
Let $u_b:=(v_b,w_b):=(\cDN)^{-1}(g_v,g_w)$, and let
$(u,\pi)$ be the solution of \eqref{SPD}.
Thanks to Theorem~\ref{th:4.1} and Proposition~\ref{pro:3.3},
$(u,\pi)$ satisfies the regularity assertion of the Corollary, and
it is the unique solution of \eqref{linStokes-homogen}
due to the definition of $\cDN.$
\end{proof}
\begin{remark}
The representation formulas in \eqref{v_2}--\eqref{v_1}
have also been derived and used by other authors,
see for instance \cite{Deni94, SS03}.
However, the optimal regularity results
in Theorem~\ref{th:3.1}, Proposition~\ref{pro:3.3},
Theorem~\ref{th:4.1}, and Corollary~\ref{co:4.2} are new. 
Moreover, the computations and arguments 
leading to these results are shorter
than in \cite{Deni94} (which only deals with the case $p=2$)
and in \cite{SS03}. We should mention, however,
that these authors consider more general domains.
\end{remark}
\section{The Linearized Two-Phase Stokes Problem with Free boundary}
In this section we consider the full linearized problem
\begin{equation}
\label{linFB}
\left\{
\begin{aligned}
\rho\partial_tu -\mu\Delta u+\nabla \pi&=f
    &\ \hbox{in}\quad &\dot\R^{n+1}\\
{\rm div}\,u&= f_d&\ \hbox{in}\quad &\dot\R^{n+1}\\
-[\![\mu\partial_y v]\!] -[\![\mu\nabla_{x}w]\!]&=g_v
    &\ \hbox{on}\quad &\R^n\\
-2[\![\mu\partial_y w]\!] +[\![\pi]\!] -\sigma\Delta h &=g_w
    &\ \hbox{on}\quad &\R^n\\
[\![u]\!] &=0 &\ \hbox{on}\quad &\R^n\\
\partial_th-\gamma w &=g_h &\ \hbox{on}\quad &\R^n\\
u(0)=u_0,\; h(0)&=h_0. &\\
 \end{aligned}
\right.
\end{equation}

We are interested in the same regularity classes for $u$ and $\pi$ as before. Then the equation for the height function
$h$ lives in the trace space
$
W^{1-1/2p}_p(J;L_p(\R^{n}))\cap L_p(J;W^{2-1/p}_p(\R^n)),
$
hence the natural space for
$h$ is given by
$$ h\in W^{2-1/2p}_p(J;L_p(\R^{n}))\cap H^1_p(J;W^{2-1/p}_p(\R^n))
\cap L_p(J;W^{3-1/p}_p(\R^n)).$$
Our next theorem states that problem (\ref{linFB}) admits maximal regularity, in particular defines an isomorphism between
the solution space and the space of data.
\begin{theorem}
\label{th:5.1}
Let $1<p<\infty$ be fixed, $p\neq 3/2,3$, and assume that $\rho_j$ and $\mu_j$ are positive
constants for $j=1,2$, and set $J=[0,a]$.
Then the  Stokes problem with free boundary \eqref{linFB} admits a unique solution $(u,\pi,h)$ with regularity
\begin{equation}
\label{reg}
\begin{split}
&u\in H^1_p(J;L_p(\R^{n+1},\R^{n+1}))
  \cap L_p(J;H^2_p(\dot{\R}^{n+1},\R^{n+1})), \\
& \pi\in L_p(J;\dot{H}^1_p(\dot{\R}^{n+1})), \\
&[\![\pi]\!]\in W^{1/2-1/2p}_p(J;L_p(\R^{n}))\cap L_p(J;W^{1-1/p}_p(\R^{n})),\\
& h\in W^{2-1/2p}_p(J;L_p(\R^{n}))\cap H^1_p(J;W^{2-1/p}_p(\R^n))
\cap L_p(J;W^{3-1/p}_p(\R^n))
\end{split}
\end{equation}
if and only if the data
$(f,f_d,g,g_h,u_0,h_0)$
satisfy the following regularity and compatibility conditions:
\begin{itemize}
\item[(a)]
$f\in L_p(J;L_p(\R^{n+1},\R^{n+1}))$,
\vspace{1mm}
\item[(b)]
$f_d\in H^1_p(J; \dot{H}^{-1}_p(\R^{n+1}))\cap L_p(J; H^1_p(\dot{\R}^{n+1}))$,
\vspace{1mm}
\item[(c)]
$g=(g_v,g_w)\in W^{1/2-1/2p}_p(J;L_p(\R^{n},\R^{n+1}))
\cap L_p(J;W^{1-1/p}_p(\R^{n},\R^{n+1}))$,
\vspace{1mm}
\item[(d)]
$g_h\in W^{1-1/2p}_p(J;L_p(\R^{n}))\cap L_p(J;W^{2-1/p}_p(\R^{n}))$,
\vspace{1mm}
\item[(e)]
$u_0\in W^{2-2/p}_p(\dot{\R}^{n+1},\R^{n+1})$, $h_0\in W^{3-2/p}_p(\R^n)$,
\vspace{1mm}
\item[(f)]
${\rm div}\, u_0=f_d(0)$ in $\,\dot\R^{n+1}$ and $[\![u_0]\!]=0$
on $\,\R^n$ if $p>3/2$,
\vspace{1mm}
\item[(g)]
$-[\![\mu\partial_y v_0]\!] -[\![\mu\nabla_{x}w_0]\!] ={g_v}(0)$ on
$\,\R^n$ if $p>3$.
\end{itemize}
The solution map $[(f,f_d,g,g_h, u_0,h_0)\mapsto (u,\pi,h)]$ is continuous between the corresponding spaces.
\end{theorem}
\begin{proof}
Similarly as in the proof of Thereom~\ref{th:3.1} we will
reduce system~\eqref{linFB} to the case where
$(f,f_d,g,u_0,h_0)=(0,0,0,0,0)$ and $g_h(0)=0$.
The Neumann-to-Dirichlet operator will once again play an essential
role in order to treat the resulting reduced problem.
\smallskip\\
(i) Let
$$
h_1(t):=
[2e^{-D_n^{1/2}t}-e^{-2D^{1/2}_nt}]h_0+ (1+D_n)^{-1}[e^{-(1+D_n)t}
-e^{-2(1+D_n)t}](g_h(0)+\gamma w_0),
$$
where $u_0=(v_0,w_0)$ and $\gamma:\R^{n+1}_{\pm}\to \R^n$ is the
trace operator.
The function $h_1$ has the following  properties
\begin{equation}
\label{tilde-h}
\begin{split}
&h_1\in W^{1/2-1/2p}_p(J;H^2_p(\R^n))\cap L_p(J;W^{3-1/p}_p(\R^n))\\
&\qquad \cap
 W^{2-1/2p}_p(J;L_p(\R^{n})) \cap H^1_p(J;W^{2-1/p}_p(\R^n)), \\
&h_1(0)=h_0,\quad \partial_t h_1(0)=g_h(0)+\gamma w_0,
\end{split}
\end{equation}
see  \cite[Lemma 6.4]{PSS07} for a proof of a similar result.
Let then $(u_1,\pi_1)$ be the solution of problem~\eqref{linStokes},
with $g_w$ replaced by $g_w+ \sigma\Delta h_1$.
It follows from Theorem~\ref{th:3.1}, the assumptions on $g=(g_v, g_w)$,
and from the first line in \eqref{tilde-h} that $(u_1,\pi_1)$ satisfies
the regularity properties stated in Theorem~\ref{th:5.1}.
\smallskip\\
(ii)
Next we consider the reduced problem
\begin{equation}
\label{linFB0}
\left\{
\begin{aligned}
\rho\partial_tu -\mu\Delta u+\nabla \pi&=0
    &\ \hbox{in}\quad &\dot\R^{n+1}\\
{\rm div}\,u&= 0&\ \hbox{in}\quad &\dot\R^{n+1}\\
-[\![\mu\partial_y v]\!] -[\![\mu\nabla_{x}w]\!]&=0
    &\ \hbox{on}\quad &\R^n\\
-2[\![\mu\partial_y w]\!] +[\![\pi]\!] -\sigma\Delta h &=0
    &\ \hbox{on}\quad &\R^n\\
[\![u]\!] &=0 &\ \hbox{on}\quad &\R^n\\
\partial_th-\gamma w&={\tilde g}_h &\ \hbox{on}\quad &\R^n\\
u(0)=0,\; h(0)&=0, \\
\end{aligned}
\right.
\end{equation}
with ${\tilde g}_h:=g_h-(\partial_t h_1-\gamma w_1)$,
where $u_1=(v_1,w_1)$ is the solution obtained in step~(i).
We conclude from \eqref{tilde-h} and the regularity properties
of $\gamma w_1$ that
\begin{equation}
\label{regularity-h-tilde}
{\tilde g}_h\in {_0}W^{1-1/2p}_p(J;L_p(\R^n))\cap L_p(J;W^{2-1/p}_p(\R^n)).
\end{equation}
Suppose that problem~\eqref{linFB0} admits a solution $(u_2,\pi_2,h_2)$   with the regularity properties stated in \eqref{reg}.
One readily verifies that
$(u,\pi,h):=(u_1+u_2,\pi_1+\pi_2,h_1+h_2)$
is a solution of problem~\eqref{linFB} in the regularity class
of \eqref{reg}.
\smallskip\\
(iii) It thus remains to show that the reduced
problem~\eqref{linFB0} admits a unique solution $(u,\pi,h)$
in the regularity class stated in Theorem~\ref{th:5.1}.
We note that once $h$ has been determined,
Corollary~\ref{co:4.2} yields the corresponding pair $(u,\pi)$
in problem~\eqref{linFB0}.
\\
To determine $h$ we extract the {\em boundary symbol} for this problem as follows. Applying the
Neumann-to-Dirichlet operator $(\cDN)^{-1}$ to $(g_v,g_w)=(0,\sigma D_n h)$
yields $\gamma u=u_b$, the trace of $u$.
According to \eqref{vw-symbol}, the tranform of the normal component
$\gamma w=w_b$ of $u_b$ is given by
\begin{equation*}
\label{hat-w-b}
 \hat w_b=
 \frac{-\sigma|\xi|^2}{(\rho_1+\rho_2)\lambda/|\xi|+4\eta_1\eta_2/(\eta_1+\eta_2)}\hat h.
\end{equation*}
Let us now consider the equation $\partial_th-\gamma w={\tilde g}_h$.
Inserting this expression for $\hat w_b$ into the transformed equation
$\lambda\hat h-\hat w_b=\hat{\tilde g}_h$ results in
$
s(\lambda,|\xi|)\hat h=\hat {\tilde g}_h
$
where the boundary symbol $s(\lambda,|\xi|)$ is
given by
\begin{equation}
\label{symbol}
s(\lambda,|\xi|)=\lambda+\frac{\sigma|\xi|^2}{(\rho_1+\rho_2)\lambda/|\xi|+4\eta_1\eta_2/(\eta_1+\eta_2)}.
\end{equation}
The operator corresponding to this symbol is
\begin{equation}
\label{S}
S= G + \sigma D_n N^{-1},
\end{equation}
where the meaning of the operators $G, D_n$ and $N$ is as in Section 4.
$S$ has the following mapping properties:
\begin{equation}
\label{mapping-S}
S: {_0H}^{r+1}_p(J;K^s_p(\R^n))\cap {_0}H^r_p(J;K^{s+1}_p(\R^n))
\to {_0}H^r(J;K^s_p(\R^n)),
\end{equation}
where $K\in\{H,W\}$.
In order to find $h$ we need to solve the equation $Sh={\tilde g}_h$,
that is, we need to show that $S$ is invertible in appropriate
function spaces.

All operators in the definition of $S$ commute,
and admit an $H^\infty$-calculus. The
$H^\infty$-angle of $D_n$ is zero, that of $N$ is $\pi/2$ and that of $G$ is $\pi/2$ as well. Thus we can a-priori not guarantee that the sum
of the power-angles of the single operators in $S$ is strictly
less than $\pi$, and the Dore-Venni approach is therefore
not directly applicable.
We will instead apply a result of
Kalton and Weis \cite[Theorem~4.4]{KW01}.

For this purpose note that for complex numbers
$w_j$ with ${\rm arg}\, w_j\in[0,\pi/2)$, we have
${\rm arg}\,(w_1w_2)/(w_1+w_2)={\rm arg}\,(1/w_1+1/w_2)^{-1}\in[0,\pi/2)$
as well. This implies that $s(\lambda,|\xi|)$ has strictly
positive real part for each $\lambda$ in the closed right halfplane and for each $\xi\in\R^n$, $(\lambda,\xi)\neq(0,0)$,
 hence $s(\lambda,|\xi|)$
does not vanish for such $\lambda$ and~$\xi$.

We write $s(\lambda,|\xi|)$
in the following way:
\begin{equation}
\label{s-lambda-tau}
s(\lambda,\tau)=\lambda+\sigma \tau k(z),\quad z=\lambda/\tau^2,
\ \lambda\in\C,\ \tau\in \C\setminus\{0\},
\end{equation}
where $$k(z)=[(\rho_1+\rho_2)z+4(\frac{1}{\sqrt{\mu_1}\sqrt{\rho_1 z+\mu_1}+\mu_2}
+\frac{1}{\sqrt{\mu_2}\sqrt{\rho_2 z+\mu_2}+\mu_1})^{-1}]^{-1}.$$
The asymptotics of $k(z)$ are given by
\begin{equation*}
k(0)= \frac{1}{2(\mu_1+\mu_2)}, \qquad
zk(z)\rightarrow \frac{1}{\rho_1+\rho_2}
\quad\text{for $z\in \C\setminus\R_{-}$ with $|z|\rightarrow\infty$}.
\end{equation*}
This shows that for any $\vartheta\in [0,\pi)$ there is a constant $C=C(\vartheta)>0$ such that
$$|k(z)|\le \frac{C}{1+|z|},\quad z\in \bar\Sigma_{\vartheta}.$$
Hence we see that
$$|s(\lambda,|\xi|)|\le C (|\lambda|+|\xi|),\quad  {\rm Re}\,\lambda\ge0, \; \xi\in\R^n,$$
is valid for some constant $C>0$.
Next we are going to prove that for each $\lambda_0>0$
there are $\eta>0$,
$c>0$  such that
\begin{equation}
\label{estsymb}
|s(\lambda,\tau)|\ge c[|\lambda|+|\tau|], \quad \mbox{ for all }
\lambda\in \Sigma_{\pi/2+\eta},\; |\lambda|\ge\lambda_0,\; \tau\in
\Sigma_\eta.
\end{equation}
This can be seen as follows:
since Re\,$k(z)>0$ for Re\,$z\ge 0$,
by continuity of the modulus and argument we  obtain an estimate
of the form
\begin{eqnarray*}
|s(\lambda,\tau)|\ge c_0[|\lambda|+|\tau||k(z)|]
\ge c[|\lambda|+|\tau|],\quad \lambda\in \Sigma_{\pi/2+\eta},\;
\tau\in \Sigma_\eta,
\end{eqnarray*}
provided $|z|\le M$, with some $\eta>0$ and $c>0$ depending on $M$, but not on
$\lambda$ and $\tau$. On the other hand,
for $m>0$ fixed we consider  the case with
$|\lambda|\ge m|\tau|$, $|z|\ge M$. We then have
\begin{equation*}
\begin{aligned}
|s(\lambda,\tau)|\ge
|\lambda|-\sigma|\tau||k(z)|
\ge \frac{1}{2}[|\lambda|+ m|\tau|]-\sigma C|\tau|/(1+M)
\ge c[|\lambda|+|\tau|],
\end{aligned}
\end{equation*}
provided $m>2\sigma C/(1+M)$, and then by extension
\begin{eqnarray*}
|s(\lambda,\tau)|\ge c[|\lambda|+|\tau|],\quad \lambda\in \Sigma_{\pi/2+\eta},\;
\tau\in \Sigma_\eta,\; |\lambda|\ge m|\tau|,  \;|z|\ge M,
\end{eqnarray*}
provided $\eta>0$ and $c>0$ are sufficiently small.
One easily sees that the intersection point of
the curves $y=Mx^2$ and $y=mx$ in $\R^2$ has distance $d=(m/M)\sqrt{1+m^2}$ from
the origin.
By choosing $M$ large enough so that
$d\le \lambda_0 $,
\eqref{estsymb} follows by combining the two estimates.

By means of the $\cR$-boundedness of the functional calculus for $D_n$
in $K^{s}_p(\R^n)$,  cf.\ Desch, Hieber and Pr\"uss \cite{DHP01a},
we see that
\begin{equation*}
 (\lambda+D_n^{1/2})s^{-1}(\lambda,D_n^{1/2})
\end{equation*}
is of class $H^\infty$ and $\cR$-bounded on
$\Sigma_{\pi/2+\eta}\setminus B_{\lambda_0}(0)$. The operator-valued $H^\infty$-calculus for
$G=\partial_t$ on ${_0H}^r_p(J;K^s_p(\R^n))$, cf.\ Hieber and Pr\"uss \cite{HiPr98},
 implies boundedness of
\begin{equation*}
(G+D_n^{1/2})s^{-1}(G,D_n^{1/2})
\quad\text{in}\quad  {_0H}^r_p(J;K^s_p(\R^n)).
\end{equation*}
This shows that $s^{-1}(G,D_n^{1/2})$ has the following mapping properties:
\begin{equation}
\label{s-inverse}
s^{-1}(G,D_n^{1/2}):{_0H}^r_p(J;K^s_p(\R^n))\to
{_0H}^{r+1}_p(J;K^s_p(\R^n))\cap {_0}H^r_p(J;K^{s+1}_p(\R^n)).
\end{equation}
We conclude that $S$ is invertible and that $S^{-1}=s^{-1}(G,D_n^{1/2})$.
Choosing $r=0$ and $s=2-1/p$ and $K=W$ in \eqref{s-inverse} yields
\begin{equation}
\label{S1}
 S^{-1}:L_p(J;W^{2-1/p}_p(\R^n))
\to {_0H}^1_p(J;W^{2-1/p}_p(\R^n))\cap L_p(J;W^{3-1/p}_p(\R^n)).
\end{equation}
Moreover, we also obtain from \eqref{s-inverse}
\begin{equation*}
\begin{aligned}
&S^{-1}:L_p(J;L_p(\R^n))\to {_0H}^{1}_p(J;L_p(\R^n))  \\
&S^{-1}:{_0H}^1_p(J;L_p(\R^n))\to {_0H}^{2}_p(J;L_p(\R^n)).
\end{aligned}
\end{equation*}
Interpolating with the real method $(\cdot\,,\cdot)_{1-1/p,p}$
then yields
\begin{equation}
\label{S2}
S^{-1}:{_0W}^{1-1/p}_p(J;L_p(\R^n))\to
{_0W}^{2-1/p}_p(J;L_p(\R^n)).
\end{equation}
\eqref{S1}--\eqref{S2} shows that the equation $Sh={\tilde g}_h$ has for each
$\tilde{g}_h$ satisfying \eqref{regularity-h-tilde} a unique solution
$h$ in the regularity class \eqref{reg}.
\smallskip\\
(iv)
Since the function $h$ is now known
we can use Corollary~\ref{co:4.2} to determine the pair $(u,\pi)$
in problem~\eqref{linFB0}. For this we note that
\begin{equation}
\label{E}
{_0}H^1_p(J;W^{2-1/p}_p(\R^n))\cap L_p(J;W^{3-1/p}_p(\R^n))
\hookrightarrow
{_0W}^{1-1/p}_p(J;H^2_p(\R^n))
\end{equation}
see \cite[Lemma 6.2]{PSS07} for a proof.
This shows that the function $h$ determined in step (iii)
satisfies
$$
\Delta h\in{_0W}^{1/2-1/2p}(J;L_p(\R^n))\cap L_p(J;W^{1-1/p}_p(\R^n))
$$
and Corollary~\ref{co:4.2} yields a solution $(u,\pi)$
in the regularity class~\eqref{reg}.
\smallskip\\
(v)
Steps (i)--(iv) render a solution $(u,\pi,h)$
for problem~\ref{linFB} that satisfies
the regularity properties asserted in the Theorem.
It follows from step (iv) and from Theorem~\ref{th:3.1}
that problem~\eqref{linFB0} with
$(f,f_d,g,g_h,u_0,h_0)=(0,0,0,0,0,0)$ has
only the trivial solution,
and this gives uniqueness.
The proof of Theorem 5.1 is now complete.
\end{proof}
\begin{remark}
Further mapping properties of the symbol $s(\lambda,\tau)$
and the associated operator $S$ have been derived in
\cite{PrSi08}. In particular, we have investigated
the singularities and zeros of the boundary symbol $s$, and
we have studied the mapping properties of $S$ in case
of low and high frequencies, respectively.
\end{remark}
\section{The nonlinear problem}
In this section we derive estimates for the nonlinear mappings
occurring on the right hand side of \eqref{tfbns2}.
In order to facilitate this
task, we first introduce some notation, and then study the
mapping properties of the nonlinear functions appearing on the
right hand sight of equation~\eqref{tfbns2}. In the following we
set
\begin{equation}
\begin{split}
&\EE_1(a):= \{u\in H^1_p(J;L_p(\R^{n+1},\R^{n+1}))
\cap L_p(J;H^2_p(\dot\R^{n+1},\R^{n+1})):\: [\![u]\!]=0\}, \\
&\EE_2(a):=L_p(J;\dot H^1_p(\dot\R^{n+1})), \\
&\EE_3(a):=
W^{1/2-1/2p}_p(J;L_p(\R^n)) \cap L_p(J;W^{1-1/p}_p(\R^n)), \\
&\EE_4(a):= W^{2-1/2p}_p(J;L_p(\R^{n}))
\cap H^1_p(J;W^{2-1/p}_p(\R^n))\\
&\hspace{1.3cm}\cap W^{1/2-1/2p}_p(J;H^2_p(\R^n))
\cap L_p(J;W^{3-1/p}_p(\R^n)),\\
&\EE(a)_{\phantom{3}}:= \{(u,\pi,q,h) \in \EE_1(a)\times
\EE_2(a)\times \EE_3(a)\times\EE_4(a):\: [\![\pi]\!]=q\}.
\end{split}
\end{equation}
The space $\EE(a)$ is given the natural norm
\begin{equation*}
\|(u,\pi,q,h)\|_{\EE(a)}
=\|u\|_{\EE_1(a)}+\|\pi\|_{\EE_2(a)}+\|q\|_{\EE_3(a)}+\|h\|_{\EE_4(a)}
\end{equation*}
which turns it into a Banach space.
We remind that $\EE_2(a)$ is equipped with the norm
$
\|\pi\|_{\EE_2(a)}=(\sum_{j=1}^{n+1}
\|\partial_j\pi\|^p_{L_p(J,L_p(\dot\R^{n+1}))})^{1/p}
$
for $\pi:\dot\R^{n+1}\to\R$.

In addition, we define
\begin{equation}
\begin{split}
&\FF_1(a):=L_p(J;L_p(\R^{n+1},\R^{n+1})), \\
&\FF_2(a):=H^1_p(J;\dot H^{-1}_p(\R^{n+1}))\cap L_p(J;H^1_p(\dot\R^{n+1})), \\
&\FF_3(a):=W^{1/2-1/2p}_p(J;L_p(\R^n,\R^{n+1}))
\cap  L_p(J;W^{1-1/p}_p(\R^n,\R^{n+1})), \\
&\FF_4(a):= W^{1-1/2p}_p(J;L_p(\R^{n}))\cap L_p(J;W^{2-1/p}_p(\R^{n})),\\
&\FF(a)_{\phantom{3}}:=\FF_1(a)\times \FF_2(a)\times
\FF_3(a)\times \FF_4(a).
\end{split}
\end{equation}
The generic elements of $\FF(a)$ are the functions $(f,f_d,g,g_h)$.

We list some properties of the function spaces
introduced above that will be used in the sequel.
In the following we say that a function space is a multiplication algebra
if it is a Banach algebra under the operation
of multiplication.
\begin{lemma}
\label{le:6.1} Suppose $p>n+3$ and let $J=[0,a]$.
Then
\begin{itemize}
\item[(a)]
$\EE_3(a)$ and $\FF_4(a)$ are multiplication algebras.
\vspace{1mm}
\item[(b)] 
$\EE_1(a) \hookrightarrow C(J;BU\!C^1(\dot\R^{n+1},\R^{n+1}))
\cap C(J;BU\!C(\R^{n+1},\R^{n+1}))$.
\vspace{1mm}
\item[(c)]
$\EE_3(a)\hookrightarrow C(J;BU\!C(\R^n))$.
\vspace{1mm}
\item[(d)] 
$
\EE_4(a) \hookrightarrow BC^1(J;BC^1(\R^n))\cap
BC(J;BC^2(\R^n)).
$ 
\vspace{1mm}
\item[(e)]
$W^{2-1/2p}_p(J;L_p(\R^{n}))
\cap H^1_p(J;W^{2-1/p}_p(\R^n))
\cap L_p(J;W^{3-1/p}_p(\R^n))
\hookrightarrow \EE_4(a)$.
\end{itemize}
\end{lemma}
\begin{proof}
(a)
The assertion that $\EE_3(a)$ and $\FF_4(a)$ are multiplication algebras
can be shown as in the proof of \cite[Lemma 6.6(ii)]{PSS07}.
\smallskip\\
(b)
It follows from \cite[Theorem III.4.10.2]{Am95}
that $\EE_1(a)\hookrightarrow C(J;W^{2-2/p}_p(\dot\R^{n+1},\R^{n+1}))$
and this implies the first inclusion,
thanks to Sobolev's embedding theorem.
The second assertion follows from the fact 
that $u$ is continuous across $y=0$.
\smallskip\\
(c)
This follows from \cite[Remark 5.3(d)]{EPS03}
and Sobolev's embedding theorem.
\smallskip\\
(d) We infer from \cite[Theorem III.4.10.2]{Am95}
that 
$$H^1_p(J;W^{2-1/p}_p(\R^n))\cap L_p(J;W^{3-1/p}_p(\R^n))
\hookrightarrow C(J;W^{3-2/p}_p(\R^n)),
$$
and the inclusion $\EE_4(a)\hookrightarrow C(J;BC^2(\R^n))$
then follows from Sobolev's embedding theorem.
In addition, we conclude from \cite[Remark 5.3(d)]{PSS07} 
and Sobolev's embedding theorem that
$$W^{1-2/p}_p(J;L_p(\R^n))\cap L_p(J;W^{2-1/p}_p(\R^n))
\hookrightarrow BC(J;BC^1(\R^n))
$$
and this implies that
$\EE_4(a) \hookrightarrow BC^{1}(J;BC^1(\R^n)).$
\smallskip\\
(e) This follows from \eqref{E}.
\end{proof}
\noindent
Let
\begin{equation}
\label{K} N(u,\pi,q,h):=\big(F(u,\pi,h),F_d(u,h), G(u,q,h), H(u,h)\big)
\end{equation}
for $(u,\pi,q,h)\in\EE(a)$, where as before $u=(v,w)$, $F=(F_v,F_w)$ and
$G=(G_v,G_w)$. We show that the mapping $N$ is real analytic.
\goodbreak
\begin{proposition}
\label{pro:estimates-K} Suppose $p>n+3$. Then
\begin{equation}
\label{K-analytic}
N\in C^\omega(\EE(a)\,,\FF(a))\quad\text{and}\quad N(0)=0,\ DN(0)=0,
\end{equation}
where $DN$ denotes the Fr\'echet derivative of
$N$.
In addition we have
\begin{equation*}
DN(u,\pi,q,h)\in \Li({_0}\EE(a),{_0}\FF(a))
\quad\text{for any}\ (u,\pi,q,h)\in\EE(a).
\end{equation*}
\end{proposition}
\begin{proof}
We first note that the
mapping $[(u,\pi,q,h)\mapsto N(u,\pi,q,h)]$ 
is polynomial. It thus suffices to
verify that $N:\EE(a)\to \FF(a)$ is well-defined and
continuous.
\smallskip\\
(i) We first consider the term $F(u,\pi,h)$,
and observe that it contains the expressions
$\nabla h,\Delta h$ and $\partial_t h$.
Without changing notation we here consider
the extension of $h$ 
from $\R^n$ to $\R^{n+1}$ defined by
$h(t,x,y)=h(t,x)$ for $t\in J$ and $(x,y)\in\R^n\times\R$.
With this interpretation we clearly have
\begin{equation}
\label{F-1}
\begin{split}
\|\partial  h\|_{\infty, J\times \R^{n+1}}
=\|\partial  h\|_{\infty, J\times \R^{n}},
\quad h\in\EE(a),
\ \partial\in\{\partial_j,\Delta,\partial_t\},
\end{split}
\end{equation}
where $\|\cdot\|_{\infty,U}$
denotes the sup-norm for the set $U\subset J\times \R^{n+1}$.
Next we note that
\begin{equation}
\label{F-2}
\begin{split}
& BC(J;BC(\R^{n+1}))\cdot L_p(J;L_p(\R^{n+1}))
\hookrightarrow L_p(J;L_p(\R^{n+1})),\\
& BC(J;BC(\R^{n+1}))\cdot BC(J;BC(\R^{n+1}))
\hookrightarrow BC(J;BC(\R^{n+1})),
\end{split}
\end{equation}
that is, multiplication is continuous and bilinear in the
indicated function spaces.
We can now conclude from \eqref{F-1}--\eqref{F-2} 
and Lemma~\ref{le:6.1} that 
\begin{equation*}
F\in C^\omega(\EE_1(a)\times\EE_2(a)\times\EE_4(a),\FF_1(a)),
\quad F(0)=0,\ DF(0)=0.
\end{equation*} 
(ii)
We will now  consider the nonlinear function
$F_d(u,h)=(\nabla h| \partial_y v).$
Since $h$ does not depend on $y$ we have
\begin{equation}
\label{Fd-1} F_d(u,h)=(\nabla h| \partial_yu)
=\partial_y (\nabla h | u).
\end{equation}
Observing that
\begin{equation*}
\label{Fd-3}
\begin{split}
&BC^1(J;BC(\R^{n+1}))\cdot H^1_p(J;L_p(\R^{n+1}))
\hookrightarrow H^1_p(J;L_p(\R^{n+1})), \\
&BC(J;BC^1(\dot\R^{n+1}))\cdot L_p(J;H^1_p(\dot\R^{n+1}))
\hookrightarrow L_p(J;H^1_p(\dot\R^{n+1})),
\end{split}
\end{equation*}
and
\begin{equation*}
\label{Fd-2}
\begin{split}
&\partial_y\in {\Li}
\big(H^1_p(J;L_p(\R^{n+1})),H^1_p(J;H^{-1}_p(\R^{n+1}))\big)\\
&\qquad\cap {\Li}
\big(L_p(J;H^1_p(\dot\R^{n+1})),L_p(J,L_p(\R^{n+1}))\big),
\end{split}
\end{equation*}
we infer from Lemma~\ref{le:6.1}(d) that
\begin{equation*}
F_d\in C^\omega(\EE_1(a)\times\EE_4(a),\FF_2(a)),
\quad F_d(0)=0,\quad DF_d(0)=0.
\end{equation*}
\noindent (iii) We remind that
\begin{equation}
\label{jump} [\![\mu \partial _i\;\cdot]\!]
\in{\Li}\big(H^1_p(J;L_p(\R^{n+1}))\cap
L_p(J;H^2_p(\dot\R^{n+1})),\EE_3(a)\big)
\end{equation}
where $[\![\mu \partial _i u]\!]$ denotes the jump of the quantity
$\mu\partial _i u$ with $u$ a generic function from
$\dot\R^{n+1}\to\R$, and where $\partial _i =\partial_{x _i} $ for
$i=1,\ldots,n$ and $\partial_{n+1} =\partial_y$.
\smallskip\\
The mapping $G(u,q,h)$ is made up of terms of the form
\begin{equation*}
[\![\mu \partial_i u_k]\!] \partial_j h, \quad [\![\mu\partial_i
u_k]\!] \partial_j h\partial_l h, \quad q\partial_j h, \quad
\Delta h\partial_j h, \quad G_\kappa(h), \quad
G_\kappa(h)\partial_jh
\end{equation*}
where $u_k$ denotes the $k$-th component of a function
$u\in\EE_1(a)$. From 
\eqref{jump} and the fact that $\EE_3(a)$ is a multiplication
algebra follows that the mappings
\begin{equation*}
\begin{split}
&(u,h)\mapsto [\![\mu \partial_i u_k]\!] \partial_j h,\
[\![\mu\partial_i u_k]\!] \partial_j h\partial_l h
:\EE_1(a)\times\EE_4(a)\to\EE_3(a), \\
& (q,h)\mapsto q\partial_jh :\EE_3(a)\times\EE_4(a)\to \EE_3(a),
\quad h\mapsto \Delta h\partial_jh:\EE_4(a)\to \EE_3(a)
\end{split}
\end{equation*}
are multilinear and continuous, and hence real analytic.
The fact that $\EE_3(a)$ is an algebra additionally implies that the
mapping $[h\mapsto G_\kappa(h)]:\EE_4(a)\to\EE_3(a)$ is analytic.
In summary we conclude that
\begin{equation*}
G\in C^\omega(\EE_1(a)\times\EE_3(a)\times\EE_4(a),\EE_3(a)),
\quad G(0)=0,\ DG(0)=0.
\end{equation*}
(iv)
We infer from
$
\gamma\in {\Li}\big(H^1_p(J;L_p(\R^{n+1}))\cap
L_p(J;H^2_p(\dot\R^{n+1})),\FF_a(a)\big)
$
and the fact that $\FF_4(a)$ is an algebra
that the mapping 
$[(u,h)\mapsto (\nabla h|\gamma u)]:\EE_1(a)\times\EE_4(a)\to \FF_4(a)$
is bilinear and continuous. This immediately yields
\begin{equation*}
H\in C^\omega(\EE_1(a)\times\EE_4(a),\FF_4(a)),
\quad H(0)=0,\ DH(0)=0.
\end{equation*}
(v)
As the terms of $N$ are made up of products
of $u,\pi,q,h$ and derivatives thereof, one easily verifies that
$$
DN(u,\pi,q,h)[\bar u,\bar \pi,\bar q,\bar h]\in {_0}\FF(a)
\ \text{whenever}\  
(u,\pi,q,h)\in \EE(a),\ (\bar u,\bar \pi,\bar q,\bar h)\in {_0}\EE(a).
$$
Combining the results obtained in steps (i)--(v)
yields the assertions of the proposition.
\end{proof}

We are now ready to prove our main result of this section,
yielding existence and uniqueness of solutions for the nonlinear
problem~\eqref{tfbns2}.
\begin{theorem}
\label{th:nonlinearII}
{\rm (}Existence of solutions for the
nonlinear problem \eqref{tfbns2}{\rm)}.
\begin{itemize}
\item[(a)] For every $t_0>0$ there exists a number
$\eps=\eps(t_0)>0$ such that for all initial values
$$(u_0,h_0)\in W^{2-2/p}_p(\dot\R^{n+1},\R^{n+1})
\times W^{3-2/p}_p(\R^n),
\quad [\![u_0]\!]=0,
$$
satisfying the compatibility conditions
\begin{equation}
\label{compatibility-II} 
[\![\mu\De(u_0,h_0)\nu_0-\mu(\nu_0|\De(u_0,h_0)\nu_0)\nu_0]\!]=0, 
\ \ {\rm div}\,u_0=F_d(u_0,h_0), 
\ \ [\![u_0]\!]=0
\end{equation}
and the smallness condition
\begin{equation}
\label{sm-II}
\|u_0\|_{W^{2-2/p}_p(\dot\R^{n+1})} 
+ \|h_0\|_{W^{3-2/p}_p(\R^n)}\le \eps,
\end{equation}
where $\De(u,h)$ is defined in \eqref{D},
the nonlinear problem \eqref{tfbns2}
admits a unique solution $(u,\pi,[\![\pi]\!],h)\in \EE(t_0)$.
\vspace{1mm}
\item[(b)] The solution has the additional regularity properties
\begin{equation*}
(u,\pi)\in C^\omega((0,t_0)\times\dot\R^{n+1},\R^{n+2}),\quad
[\![ \pi ]\!],h\in C^\omega((0,t_0)\times\R^n). \\
\end{equation*}
In particular,
${\mathcal M}=\bigcup_{t\in (0,t_0)}\big(\{t\}\times\Gamma(t)\big) \text{ is a real analytic manifold}$.
\end{itemize}
\end{theorem}
\begin{proof}
In order to  economize our notation we set $z:=(u,\pi,q,h)$ for
$(u,\pi,q,h)\in\EE(a)$. With this notation, the nonlinear
problem~\eqref{tfbns2} can be restated as
\begin{equation}
\label{FP-1} Lz=N(z),\quad (u(0),h(0))=(u_0,h_0),
\end{equation}
where $L$ denotes the linear operator on the left-hand side of
~\eqref{tfbns2}, and where $N$ is defined in \eqref{K}.

It is convenient to first introduce an auxiliary function
$z^\ast\in\EE(a)$ which resolves
the compatibility conditions
\eqref{compatibility-II} and the initial conditions in
\eqref{FP-1}, and then to solve the resulting reduced
problem
\begin{equation}
\label{FP-2} 
Lz=N(z+z^\ast)-Lz^\ast=:K_{0}(z), \quad z\in {_0}\EE(a),
\end{equation}
by means of a fixed point argument.
\smallskip\\
(i) Suppose that the initial values $(u_0,h_0)$ satisfy the
(first) compatibility condition in \eqref{compatibility-II}, and set
\begin{equation*}
[\![\pi_0]\!]:=[\![\mu(\nu_0 | \De(u_0,h_0)\nu_0)]\!]
+\sigma(\Delta h_0-G_\kappa(h_0)).
\end{equation*}
It is then clear that the following compatibility conditions hold:
\begin{equation}
\label{comp-2}
\begin{aligned}
-[\![\mu\partial_y v_0]\!] -[\![\mu\nabla_{x}w_0]\!]
 &=G_v(u_0,[\![\pi_0]\!],h_0)
    &\ \hbox{on}\quad &\R^n\\
-2[\![\mu\partial_y w_0]\!] +[\![\pi_0]\!] -\sigma\Delta h_0 &=
G_w(u_0,h_0)
    &\ \hbox{on}\quad &\R^n\\
\end{aligned}
\end{equation}
where $u_0=(v_0,w_0)$. Next we introduce special
functions $(0,f^\ast_d,g^\ast,g^\ast_h)\in\FF(a)$ which resolve
the necessary compatibility conditions. First we set
\begin{equation}
\label{c-star} c^\ast(t):= \left\{
\begin{aligned}
&{\mathcal R}_{+} e^{-t D_{n+1}}\mathcal{E}_{+}(v_0|\nabla
h_0)
\quad\text{in}\quad \R^{n+1}_{+}, \\
&{\mathcal R}_{-} e^{-t D_{n+1}}\mathcal{E}_{-}(v_0|\nabla h_0)
\quad\text{in}\quad \R^{n+1}_{-},
\end{aligned}
\right.
\end{equation}
where ${\mathcal E}_{\pm}\in
\Li(W^{2-2/p}_p(\R^{n+1}_{\pm}),W^{2-2/p}_p(\R^{n+1}))$ is an
appropriate extension operator and ${\mathcal R}_{\pm}$ is the
restriction operator. Due to $(v_0|\nabla h_0)\in
W^{2-2/p}_p(\dot\R^{n+1})$ we obtain
\begin{equation*}
c^\ast\in H^1_p(J;L_p(\R^{n+1}))\cap
L_p(J;H^2_p(\dot\R^{n+1})).
\end{equation*}
Consequently,
\begin{equation}
\label{divergence-star}
 f^\ast_d:=\partial_y\, c^\ast\in \FF_2(a)
\ \text{ and }\  f^\ast_d(0)=F_d(v_0,h_0).
\end{equation}
Next we set
\begin{equation}
\label{star}
\begin{split}
g^\ast(t):=e^{-D_nt} G(u_0,[\![\pi_0]\!],h_0), \quad
g^\ast_h(t):=e^{-D_nt}H(u_0,h_0).
\end{split}
\end{equation}
It then follows from
\eqref{divergence-star} and \cite[Lemma 8.2]{EPS03} that
$(0,f^\ast_d,g^\ast,g^\ast_h)\in\FF(a)$. \eqref{comp-2} and the
second and third condition in \eqref{compatibility-II} show
that the necessary compatibility conditions of
Theorem~\ref{th:5.1}
are satisfied and we can conclude that the
linear problem
\begin{equation}
\label{FP-3} Lz^\ast=(0,f^\ast_d,g^\ast,g^\ast_h), \quad
(u^\ast(0),h^\ast(0))=(u_0,h_0),
\end{equation}
has a unique solution $z^\ast\in\EE(a)$.
 With the auxiliary function $z^\ast$ now determined, we can
focus on the reduced equation \eqref{FP-2}, which can be converted
into the fixed point equation
\begin{equation}
\label{FP-4} z=L_{0}^{-1}K_{0}(z),\quad z\in{_0}\EE(a),
\end{equation}
where $L_0$ denotes the restriction of $L$ to ${_0}\EE(a)$.
 Due to the choice of
$(f^\ast_d,g^\ast,g^\ast_h)$ we have $K_{0}(z)\in {_0}\FF(a)$ for
any $z\in{_0}\EE(a)$, and it follows from
Proposition~\ref{pro:estimates-K} that
\begin{equation*}
\label{FP-6} K_{0}\in C^\omega({_0}\EE(a),{_0}\FF(a)).
\end{equation*}
Consequently, $L_{0}^{-1}K_{0}:{_0}\EE(a)\to {_0}\EE(a)$ is well defined
and smooth.
\medskip\\
(ii) 
In the following, $t_0>0$ is a fixed number.
We set
\begin{equation*}
\label{E_1}
E_1:=\{(u_0,h_0)\in W^{2-2/p}_{p}(\dot\R^{n+1},\R^{n+1})\times W^{3-2/p}_p(\R^n):
[\![u]\!]=0\},
\end{equation*}
and observe that $E_1$ is a Banach space.
Given $(u_0,h_0)\in E_1$ let $(f^\ast_d, g^\ast,g^\ast_h)$ be defined as in
 \eqref{divergence-star}--\eqref{star}.
It is not difficult to see that the mapping
\begin{equation*}
F^\ast: E_1\to \FF(t_0),
\quad
F^\ast(u_0,h_0):=(0,f^\ast_d, g^\ast,g^\ast_h),
\end{equation*}
is $C^1$ (in fact real analytic), and that
$F^\ast(0)=0$ and $DF^\ast(0)=0$.
Hence given $\delta\in (0,1)$ there exists $\eps=\eps(\delta)>0$ such that
\begin{equation}
\label{FF^ast-estimate}
\|F^\ast(u_0,h_0)\|_{\FF(t_0)}\le \delta\|(u_0,h_0)\|_{E_1},
\quad (u_0,h_0)\in \eps \overline{\BB}_{E_1}.
\end{equation}
Let  $\mathbb G(t_0)$ denote the closed subspace
of $\FF(t_0)\times E_1$ consisting of all functions
$(f,f_d,g,g_h,u_0,h_0)$ satisfying the compatibility
conditions of Theorem~\ref{th:5.1}.

Suppose that $(u_0,h_0)\in\eps\overline{\BB}_{E_1}$ satisfies the compatibility
conditions~\eqref{compatibility-II}.
Then, due to \eqref{comp-2} and the definition of
$F^\ast$, the mapping
\begin{equation*}
G^\ast:E_1 \to \GG(t_0),
\quad G^\ast(u_0,h_0):=(F^\ast(u_0,h_0),u_0,h_0),\\
\end{equation*}
is well-defined and $\|G^\ast(u_0,h_0)\|_{\GG(t_0)}
\le 2\|(u_0,h_0)\|_{E_1}$.
It then follows from Theorem~\ref{th:5.1} that
\eqref{FP-3}
has a unique solution $z^\ast=z^\ast(u_0,h_0)$ which satisfies
\begin{equation}
\label{z^ast-uniform}
\|z^\ast\|_{\EE(t_0)}\le C_0\|(u_0,h_0)\|_{E_1},
\quad (u_0,h_0)\in\eps\overline{\BB}_{E_1}\,,
\end{equation}
where the constant $C_0$ does not depend on  $(u_0,h_0)$.
\medskip\\
(iii) Theorem~\ref{th:5.1} also implies that 
$L_{0}: {_0}\EE(t_0)\to {_0}\FF(t_0)$
is an isomorphism. Let then
\begin{equation}
\label{FP-5-zero}
M:=\|L_{0}^{-1}\|_{\Li({_0}\FF(t_0),{_0}\EE(t_0))}.
\end{equation}
We can assume that the number $\delta$ in step (ii)
was already chosen sufficiently small such that
\begin{equation}
\label{choice-delta}
\delta<\min\big(1,\frac{1}{M(2+C_0)}\big).
\end{equation}
\\
(iv)
We shall show that the fixed point equation~\eqref{FP-4}
has for each initial value $(u_0,h_0)$
satisfying \eqref{compatibility-II}--\eqref{sm-II}
a unique fixed point
$\hat z=\hat z(u_0,h_0)\in \eps\overline{\BB}_{{_0}\EE(t_0)}$.
It follows from Proposition~\ref{pro:estimates-K} 
and \eqref{z^ast-uniform} that
\begin{equation}
\label{DN-estimate}
\|DN(z+z^\ast)\|_{{\Li}(\EE(t_0),\FF(t_0))}, \
\|DK_0(z)\|_{{\Li}({_0}\EE(t_0),{_0}\FF(t_0))}\le \delta
\end{equation}
for all $(u_0,h_0)$ satisfying \eqref{compatibility-II}--\eqref{sm-II}
and all $z\in \eps\overline{\BB}_{{_0}\EE(t_0)}$,
provided $\eps$ is chosen small enough.
From \eqref{FF^ast-estimate}--\eqref{DN-estimate} follows for $z, z_j\in\eps\overline{\BB}_{{_0}\EE(t_0)}$
\begin{equation*}
\begin{split}
\|L_0^{-1}(K_0(z_1)-K_0(z_2))\|_{{_0}\EE(t_0)}
\le M\delta \|z_1-z_2\|_{{_0}\EE(t_0)} 
\le (1/2)\|z_1-z_2\|_{{_0}\EE(t_0)}
\end{split}
\end{equation*}
and
\begin{equation*}
\begin{split}
\|L_0^{-1}K_0(z)\|_{{_0}\EE(t_0)}
&\le M\big(\|N(z+z^\ast)\|_{\FF(t_0)}+\|F^\ast(u_0,h_0)\|_{\FF(t_0)}\big)
\\
&\le M\delta(2+C_0)\eps\le \eps.
\end{split}
\end{equation*}
This shows that the mapping
$
L_{0}^{-1}K_{0}: \eps\overline{\BB}_{{_0}\EE(t_0)} \to
\eps\overline{\BB}_{{_0}\EE(t_0)}
$
is a contraction for any initial value
$(u_0,h_0)$ satisfying \eqref{compatibility-II}--\eqref{sm-II}.
\smallskip\\
(v)  By the contraction mapping principle $L_{0}^{-1}K_{0}$
has a unique fixed point $\hat z \in
\eps\overline{\BB}_{{_0}\EE(t_0)} \subset {_0}\EE(t_0) $ and it
follows from \eqref{FP-1}--\eqref{FP-2} that
$\hat z+z^*$ is the (unique) solution of the nonlinear
problem~\eqref{tfbns2} in
$\EE(t_0)$, proving
the assertion in part (a) of the Theorem.
\smallskip\\
(vi) In order to show that $(u,\pi,q,h)$ is analytic in space and
time we can use the same strategy as in \cite[Section~8]{EPS03}.
Since the proof is  similar  we
will refrain from giving all the details, and will rather point
out the underlying ideas.
\smallskip\\
Let $(u,\pi,q,h)\in \EE(t_0)$ be the solution of \eqref{tfbns2}
with initial value $(u_0,h_0)$. Let $a\in (0,t_0)$ be fixed and
choose $\delta>0$ so that $(1+\delta)a\le t_0$. Moreover, let
$\varphi$ be a smooth cut-off function with $\varphi\equiv1$ on
$[-R,R]$ for some $R>0$
and suppose that $\delta>0$ is chosen small enough so that
$$
1+\varphi(y)\tau t>0,\quad 1+(y\varphi(y))^\prime\tau t>0, \quad
t\in [0,a],\; \tau\in (-\delta,\delta),\; y\in\R.
$$
For given parameters $(\lambda,\nu,\tau)\in
(1-\delta,1+\delta)\times\R^n\times (-\delta,\delta)$ we set
\begin{equation}
\label{R-1}
\begin{aligned}
(u_{\lambda,\nu,\tau},\pi_{\lambda,\nu,\tau})(t,x,y):&=
(u,\pi)(\lambda t,x+t\nu,y(1+\varphi(y)\tau t)),
 \\
(q_{\lambda,\nu},h_{\lambda,\nu})(t,x):&= (q,h)(\lambda t,x+t\nu),
 \\
z_{\lambda,\nu,\tau}:&=(u_{\lambda,\nu,\tau},\pi_{\lambda,\nu,\tau},
q_{\lambda,\nu},h_{\lambda,\nu})
\end{aligned}
\end{equation}
where $(t,x,y)\in [0,a]\times\R^n\times\dot\R$.
Suppose we know that
\begin{equation}
\label{R-2} [(\lambda,\nu,\tau)\mapsto z_{\lambda,\nu,\tau}] \in
C^{\omega}(\Lambda,\EE(a))
\end{equation}
with $\Lambda\subset (1-\delta,1+\delta)\times\R^n\times
(-\delta,\delta)$ a neighborhood of $(\lambda,\nu,\tau)=(1,0,0)$.
Pick $(s_0,x_0,y_0)\in (0,t_0)\times\dot\R^{n+1}$ and choose $a\in
(s_0,t_0)$. Without loss of generality we can assume that $y_0\in
[-R,R]$. Thanks to the embeddings
$$
\EE_1(a)\hookrightarrow C(I,BC(\R^{n+1},\R^{n+1})),\quad
\EE_{3}(a),\EE_4(a)\hookrightarrow C(I;BC(\R^n)),
$$
see Lemma~\ref{le:6.1},
we conclude that
\begin{equation*}
\begin{split}
&\left[(\lambda,\nu,\tau)\mapsto u_{\lambda,\nu,\tau}\right]\in
C^\omega(\Lambda,C(I;BC(\R^{n+1},\R^{n+1})),\\
&\left[(\lambda,\nu)\mapsto
(q_{_\lambda,\nu},h_{\lambda,\nu})\right] \in C^\omega(\Lambda,
C(I;BC(\R^n)\times C(I;BC(\R^n))
\end{split}
\end{equation*}
for $I=[0,a]$. Thus
\begin{equation*}
\begin{split}
&\left[(\lambda,\nu,\tau)
\mapsto u(\lambda s_0, x_0+s_0\nu,y_0(1+\tau s_0)\right]
\in C^\omega(\Lambda,\R^{n+1}), \\
&\left[(\lambda,\nu)\mapsto (q,h)(\lambda s_0,x_0+s_0\nu)\right]
\in C^\omega(\Lambda,\R^2),
\end{split}
\end{equation*}
and this implies that
\begin{equation}
\label{R-3} u\in C^\omega ((0,t_0)\times \dot\R^{n+1},\R^{n+1}),\quad q,h\in
C^\omega((0,t_0)\times\R^n).
\end{equation}
This in turn with \eqref{2.1}--\eqref{2.2} shows that
$\nabla\pi\in C^\omega((0,t_0)\times\dot\R^{n+1},\R^{n+1})$ as well, and we can
now conclude that
\begin{equation}
\label{pi-analytic} \pi\in C^\omega((0,t_0)\times \dot\R^{n+1}),
\end{equation}
where the pressure $\pi$ is normalized by $\pi(t,0,0-)\equiv0$, i.e.
$$
\pi(t,x,y)=\left\{\begin{array}{ll}
q(t,0)+\int_0^1[(\nabla_x\pi(t,sx,sy)|x)+\partial_y\pi(t,sx,sy) y ]ds,&y>0,\\
\int_0^1[(\nabla_x\pi(t,sx,sy)|x)+\partial_y\pi(t,sx,sy) y ]ds,&y<0.
\end{array}\right.
$$
\smallskip\\
(vii) We will now explain the steps needed to establish the
crucial property~\eqref{R-2}. First we note that there exists a
neighborhood $\Lambda \subset
(1-\delta,1+\delta)\times\R^n\times(-\delta,\delta)$ of $(1,0,0)$
such that
\begin{equation}
\label{R-4} [(\lambda,\nu,\tau)\mapsto
(0,f^\ast_{d,\lambda,\nu,\tau},g^\ast_{\lambda,\nu},
g^\ast_{h,\lambda,\nu})] \in C^\omega(\Lambda,\FF(a))
\end{equation}
where the functions $(f_d^\ast,g^\ast,g^\ast_h)$ are defined in
\eqref{divergence-star}--\eqref{star}. In fact, the assertion
follows immediately from \cite[Lemma 8.2]{EPS03} for the functions
$(g^\ast,g^\ast_h)$. Let us then consider the function $c^\ast$
defined in \eqref{c-star}. Let $w(t):=e^{-tD_{n+1}}w_0$ for some
function $w_0\in W^{2-2/p}_p(\R^{n+1})$ and define
$w_{\lambda,\nu,\tau}(t,x,y)$ for $(t,x,y)\in I\times\R^{n+1}$ as
above, with $I=[0,a]$. Then one verifies as in the proof of
\cite[Lemma 8.2]{EPS03} that
$$
w_{\lambda,\nu,\tau}\in  H^1_p(I;L_p(\R^{n+1})) \cap L_p(I;
H^2_p(\R^{n+1}))=:{\mathbb X}_1(I)
$$
for $(\lambda,\nu,\tau)\in (1-\delta,1+\delta)\times\R^n\times
(-\delta,\delta)$, and that $w_{\lambda,\nu,\tau}$ solves the
parameter-dependent parabolic equation
\begin{equation*}
\label{R5}
\partial_tu -{\mathcal A}_{\lambda,\nu,\tau}u=0,
\quad u(0)=w_0,
\end{equation*}
in $\R^{n+1}$, where ${\mathcal A}_{\lambda,\nu,\tau}$ is a
parameter-dependent differential operator given by
\begin{equation*}
\begin{split}
{\mathcal A}_{\lambda,\nu,\tau}=\lambda\Delta_x
+\frac{\lambda}{(1+\alpha^\prime(y)\tau t)^2}\partial^2_y
+\tau\left(\frac{\alpha(y)}{1+\alpha^\prime(y)\tau t}
-\frac{\lambda\alpha^{\prime\prime}(y)t}{(1+\alpha^\prime(y)\tau
t)^3}\right)\partial_y+(\nu |\nabla_x )
\end{split}
\end{equation*}
for $t\in [0,a]$ and $y\in\dot\R$, where
$\alpha(y):=y\varphi(y)$. Here we observe that
\begin{equation*}
{\mathcal A}_{1,0,0}=\Delta,\quad \left[(\lambda,\nu,\tau)\mapsto
{\mathcal A}_{\lambda,\nu,\tau}\right] \in C^\omega(\Lambda,\Li({\mathbb X}_1(I),{\mathbb X}_0(I)),
\end{equation*}
with ${\mathbb X}_0(I):=L_p(I,L_p(\R^{n+1})$. As in the proof of
\cite[Lemma 8.2]{EPS03} it follows from the implicit function
theorem that there exists a neighborhood
$\Lambda\subset(1-\delta,1+\delta)\times\R^n\times
(-\delta,\delta)$ of $(1,0,0)$ such that
\begin{equation}
\label{R6} [(\lambda,\nu,\tau)\mapsto w_{\lambda,\nu,\tau}] \in
C^\omega(\Lambda,{\mathbb X}_1(I)).
\end{equation}
Applying \eqref{R6} separately to $w_0={\mathcal E}_{\pm}(v_0\nabla
h_0)$, an then applying ${\mathcal R}_{\pm}$ yields
\begin{equation*}
[(\lambda,\nu,\tau)\mapsto c^\ast_{\lambda,\nu,\tau}] \in
C^\omega(\Lambda, H^1_p(I;L_p(\R^{n+1})) \cap L_p(I;
H^2_p(\dot\R^{n+1})).
\end{equation*}
It then follows from the definition of $f^\ast_d$ that $
[(\lambda,\nu,\tau)\mapsto f^\ast_{d,\lambda,\nu,\tau}] \in
C^\omega(\Lambda,\FF_2(a)). $ In a next step one verifies that the
function $z^\ast_{\lambda,\nu,\tau}$ solves the linear
parameter-dependent problem
\begin{equation}
\label{linFB-paramter} \left\{
\begin{aligned}
\rho\partial_tu -{\mathcal A}_{\lambda,\nu,\tau}u +{\mathcal B}_{\lambda,\tau}\pi 
&=0    &\ \hbox{in}\quad &\dot\R^{n+1}\\
{\mathcal C}_{\tau}u&= f^\ast_{d,\lambda,\nu,\tau}
 &\ \hbox{in}\quad &\dot\R^{n+1}\\
-\frac{1}{1+\tau t}[\![\mu\partial_y v]\!] -[\![\mu\nabla_{x}w]\!]
&=g^\ast_{v,\lambda,\nu}
    &\ \hbox{on}\quad &\R^n\\
-\frac{2}{1+\tau t}[\![\mu\partial_y w]\!] +[\![\pi]\!]
-\sigma\Delta h &=g^\ast_{w,\lambda,\nu}
    &\ \hbox{on}\quad &\R^n\\
[\![u]\!] &=0 &\ \hbox{on}\quad &\R^n\\
\partial_th-\lambda\gamma w + {\mathcal D}_{\nu}h
&=\lambda g^\ast_{h,\lambda,\nu}  &\ \hbox{on}\quad &\R^n\\
u(0)=u_0,\; h(0)&=h_0 &\\
 \end{aligned}
\right.
\end{equation}
where
\begin{equation*}
\begin{split}
&{\mathcal A}_{\lambda,\nu,\tau}:=\lambda\mu\Delta_x
\!+\!\frac{\lambda\mu}{(1+\alpha^\prime(y)\tau t)^2}\partial^2_y
+\tau\big(\frac{\rho\alpha(y)}{1+\alpha^\prime(y)\tau t}
\!-\!\frac{\lambda\mu\alpha^{\prime\prime}(y)t}{(1+\alpha^\prime(y)\tau t)^3}\big)\partial_y+\rho (\nu|\nabla_x), \\
&{\mathcal B}_{\lambda,\tau}\pi:= \lambda (\nabla_x
\pi,\frac{1}{1+\alpha^\prime(y)\tau t}\partial_y\pi), \qquad
{\mathcal C}_\tau u:={\rm div}_x v
+\frac{1}{1+\alpha^\prime(y)\tau t}\partial_y w, \\
&{\mathcal D}_{\nu}h:=-(\nu|\nabla h).
\end{split}
\end{equation*}
We note that
\begin{equation*}
\mathcal A_{1,0,0}=\mu\Delta, \quad \mathcal B_{1,0}=\nabla,\quad
\mathcal C_1={\rm div},\quad \mathcal D_{0}=0.
\end{equation*}
It is easy to see that the differential operators $\mathcal
A_{\lambda,\nu,\tau}$, $\mathcal B_{\lambda,\tau}$, 
$\mathcal C_\tau$ and $\mathcal D_\nu$ 
depend analytically on the parameters $(\lambda,\nu,\tau)$
in the appropriate function spaces.
Using Thereom~\ref{th:5.1} and the implicit
function theorem one shows similarly as in \cite[Lemma 8.3]{EPS03}
that there is a neighborhood $\Lambda\subset
(1-\delta,1+\delta)\times\R^n\times (-\delta,\delta)$ of $(1,0,0)$ such that
\begin{equation}
\label{R-7} [(\lambda,\nu,\tau)\mapsto z^\ast_{\lambda,\nu,\tau}]
\in C^\omega(\Lambda,\EE(a)).
\end{equation}
Let $\hat z$ be the solution of  \eqref{FP-2} obtained in step (v)
above. Then one verifies that $\hat
z_{\lambda,\nu,\tau}\in 2\eps\BB_{{_0}\EE(t_0)}$ for
$(\lambda,\nu,\tau)\in\Lambda$, with $\Lambda$ a sufficiently
small neighborhood of $(1,0,0)$. 
Moreover, $\hat z_{\lambda,\nu,\tau}$ 
solves the nonlinear parameter-dependent problem
\begin{equation}
\label{L-lambda}
L_{\lambda,\nu,\tau} z=K_{\lambda,\nu,\tau}(z), \quad z\in
{_0}\EE(a),
\end{equation}
for $(\lambda,\nu,\tau)\in\Lambda$, where $L_{\lambda,\nu,\tau} z$
is defined by the left-hand side of~\eqref{linFB-paramter}
and where
\begin{equation}
\label{K-lambda}
K_{\lambda,\nu,\tau}(z):= \left(
\begin{array}{rll}
\ab \lambda F_\tau(u+u^\ast_{\lambda,\nu,\tau}\,,
\pi+\pi^\ast_{\lambda,\nu,\tau}\,, h+h^\ast_{\lambda,\nu})
\vspace{1mm}
\\
\ab F_{d,\tau}(u+u^\ast_{\lambda,\nu,\tau}\,,h+h^\ast_{\lambda,\nu})
-f^\ast_{d,\lambda,\nu,\tau} \vspace{1mm}
\\
\ab G_\tau(u+u^\ast_{\lambda,\nu,\tau}\,, q+q^\ast_{\lambda,\nu}\,,
h+h^\ast_{\lambda,\nu}) -g^\ast_{\lambda,\nu} \vspace{1mm}
\\
\ab \lambda H(u+u^\ast_{\lambda,\nu,\tau}\,, h+h^\ast_{\lambda,\nu})
-g^\ast_{h,\lambda,\nu} \vspace{1mm}
\\
\end{array}
\right).
\end{equation}
The functions $F_{\tau}$,  $F_{d,\tau}$ and $G_{\tau}$ 
are obtained from $F$, $F_d$ and $G$, respectively,
by replacing terms containing partial derivatives
$\partial_y$ and $\partial^2_y$ in the following way:
\begin{equation*}
\partial_y \omega  \mapsto \frac{1}{1+\alpha^\prime(y)\tau t}\partial_y \omega,
\quad
\partial^2_y \omega  \mapsto \frac{1}{(1+\alpha^\prime(y)\tau t)^2}\partial^2_y \omega
  -\frac{\alpha^{\prime\prime}(y)\tau t}{(1+\alpha^\prime(y)\tau t)^3}\partial_y\omega
\end{equation*}
for $\omega\in\{v,w,\pi\}.$
Equation \eqref{L-lambda} can be reformulated as
\begin{equation}
\label{Psi}
\Psi(z,(\lambda,\nu,\tau)):=z-(L_{\lambda,\nu,\tau})^{-1}K_{\lambda,\nu,\tau}(z)=0,
\quad z\in {_0}\EE(a).
\end{equation}
Here we observe that $\Psi(\hat z,(1,0,0))=0$ for the solution $\hat z$
of the fixed point equation~\eqref{FP-4}. It follows from
\eqref{R-4}, \eqref{R-7} and Proposition~\ref{pro:estimates-K}
that
\begin{equation*}
[(z,(\lambda,\nu,\tau))\mapsto \Psi(z,(\lambda,\nu,\tau))] \in
C^\omega({_0}\EE(a)\times\Lambda, {_0}\EE(a)).
\end{equation*}
Moreover, it follows from~\eqref{FP-5-zero}--\eqref{DN-estimate} that
\begin{equation*}
D_1\Psi(\hat z,(1,0,0))=I-D(L^{-1}K_{0})(\hat z)
\in\text{Isom}(_{0}\EE(a),{_0}\EE(a)).
\end{equation*}
By the implicit function theorem there exists a neighborhood
$\Lambda\subset (1-\delta,1+\delta)\times\R^n\times
(-\delta,\delta)$ of $(\lambda,\nu,\tau)=(1,0,0)$ such that
\begin{equation}
\label{R-11} 
[(\lambda,\nu,\tau)\mapsto \hat z_{\lambda,\nu,\tau}]
\in C^\omega(\Lambda, {_0}\EE(a)).
\end{equation}
Combining \eqref{R-7} and \eqref{R-11} yields \eqref{R-2}. This
completes the proof of Theorem~\ref{th:nonlinearII}.
\end{proof}
\noindent
{\bf Proof of Theorem 1.1:}
We first observe that 
the compatibility conditions of Theorem 1.1 are
satisfied if and only if \eqref{compatibility-II}
is satisfied.
Next we note that the mapping $\Theta_{h_0}$ given by
$\Theta_{h_0}(x,y):=(x,y+h_0(x))$
defines for each $h_0\in W^{3-2/p}_p(\R^n)$
a $C^2$-diffeomorphism
from $\R^{n+1}_{\pm}$ onto $\Omega_i(0)$
with
$\det [D\Theta_{h_0}(x,y)]=1$.
Its inverse is given by $\Theta^{-1}_{h_0}(x,y):=(x,y-h_0(x))$.
It then follows from the chain rule and 
the transformation rule for integrals that
\begin{equation*}
\begin{split}
\frac{1}{C(h_0)}
\|u_0\|_{W^{2-2/p}_p(\Omega_0)}
\le \|(v_0,w_0)\|_{W^{2-2/p}_p(\dot\R^{n+1})} 
\le  C(h_0)
\|u_0\|_{W^{2-2/p}_p(\Omega_0)}
\end{split}
\end{equation*}
where 
$C(h_0):=M[1+\|\nabla h_0\|_{BC^1(\R^n)}],$
with $M$ an appropriate constant.
Consequently, there exists $\eps_0>0$ such that
$
\|u_0\|_{W^{2-2/p}_p(\Omega_0)}+\|h_0\|_{W^{3-2/p}_p(\R^n)}\le \eps_0
$
implies the smallness-condition~\eqref{sm-II}.
Theorem~\ref{th:nonlinearII}
then yields a unique solution
$(v,w,\pi,[\pi],h)\in \EE(t_0)$ which satisfies the
additional regularity properties listed in part (b) of the theorem.
Setting
\begin{equation*}
(u,q)(t,x,y)=(v,w,\pi)(t,x,y-h(t,x)),\quad (t,x,y)\in{\mathcal O},
\end{equation*}
we then conclude that
$(u,q)\in C^\omega({\mathcal O},\R^{n+2})$ and $[q]\in C^\omega(\mathcal M)$.
The regularity properties listed in \eqref{h-continuous}--\eqref{u-continuous}
are implied by Lemma~\ref{le:6.1}(b)-(c).
Finally, since $\pi(t,x,y)$ is defined 
for every $(t,x,y)\in{\mathcal O}$, we can conclude that
$\pi(t,\cdot)\in \dot H^{1}_p(\Omega(t))\subset U\!C(\Omega(t))$
for every $t\in (0,t_0)$.
\hfill{$\square$}

\end{document}